\documentclass[12pt]{article}

\usepackage{amssymb,amsthm,MnSymbol,intcalc}
\usepackage{tikz}
\usepackage[dvipsnames]{xcolor}
\usepackage{subcaption}
\usepackage{graphicx}
\usetikzlibrary{calc,decorations.pathreplacing,decorations.markings,arrows,graphs}
\usepackage[margin=2cm]{geometry}
\usepackage[hidelinks]{hyperref}
\usepackage[capitalise,nameinlink]{cleveref}

\newtheorem{thm}{Theorem}[section]
\newtheorem{lem}[thm]{Lemma}
\newtheorem{claim}{Claim}

\newtheorem{cor}[thm]{Corollary}

\theoremstyle{definition}

\newcommand{\mcon}{blue}
\newcommand{\mctw}{red}
\newcommand{\mcth}{green}
\newcommand{\mcfo}{magenta}
\newcommand{\mcfi}{cyan}
\newcommand{\mcsi}{yellow}
\newcommand{\mcse}{brown}
\newcommand{\cms}{\text{cms}}

\newcommand{\fS}{\mathcal{S}}
\newcommand{\fH}{\mathcal{H}}

\newcommand{\co}{\text{co}}

\newcommand{\ola}{\overleftarrow{C}}
\newcommand{\ora}{\overrightarrow{C}}
\pgfmathsetmacro{\scale}{.4}
\pgfmathsetmacro{\hscale}{.625}
\pgfmathsetmacro{\hlscale}{1.5}
\pgfmathsetlengthmacro\sf{0.75cm}
\pgfmathsetlengthmacro\r{1.41*\sf}
\pgfmathsetlengthmacro\s{1.63*\sf}
\pgfmathsetlengthmacro\lwid{0.15*\sf}
\pgfmathsetlengthmacro\cwid{0.1*\sf}
\pgfmathsetlengthmacro\inse{0.125*\sf}  \pgfmathsetlengthmacro\inci{0.20*\sf}
\pgfmathsetmacro{\a}{60}
\author{Adam Mammoliti \thanks{
School of Mathematics and Statistics
UNSW Sydney
NSW 2052, Australia
}\\
\texttt{adam.mammoliti@outlook.com.au}
}
\title{On the Hamiltonicity, traceability, and toughness of complements of line graphs. }
\begin{document}
\maketitle
\begin{abstract}
A coline graph $\co(G)$ of a graph $G$ is the graph with vertex set $E(G)$ for which two vertices $e$ and $e'$ of $\co(G)$ are adjacent if and only if they are not adjacent as edges in $G$. A graph $G$ is tough if the number of connected components of $G-S$ is at most $|S|$ for all cut sets $S$. Wu and Meng, and Liu independently gave similar characterisations of coline graphs that are Hamiltonian.
In this paper we give an alternate proof of Wu and Meng's and Liu's results using the longest cycle method.  We in fact prove the following reformulation of their results. A tough coline graph $\co(G)$ is Hamiltonian unless $G$ is one of four examples, one of which is $K_5$, since $\co(K_5)$ is the Petersen graph. Characterisations of tough coline graphs and of coline graphs that contain a Hamiltonian path are also given.
\end{abstract}



\section{Introduction}
The line graph of a graph $G=(V,E)$, denoted $L(G)$, is the graph with vertex set $E$ such that $ee'$ is an edge in $L$ if and only if $e$ and $e'$ are adjacent in $G$. The complement $\overline{G}$ of a graph $G=(V,E)$ is the graph with vertex set $V$ such that $vw$ is an edge in $\overline{G}$ if and only if $vw$ is not an edge in $G$. 
Here we mainly consider the complement of line graphs, which we call {\em coline graphs} for short. For a graph $G$, we let $\co(G) = \overline{L(G)}$ denote its coline graph. Alternatively, the coline graph $\co(G)$ of a graph $G=(V,E)$ can be defined directly as the graph with vertex set $E$ for which $ee'$ is an edge in $\co(G)$ if and only if $e$ and $e'$ are \emph{not} adjacent in $G$. For any line graph or coline graph $L$, we call $G$ a {\em line root} or a {\em coline root} of $L$ if $L(G)=L$ or $\co(G) = L$, respectively.

A {\em Hamiltonian cycle} of a graph $G$ is a spanning cycle and a graph containing a Hamiltonian cycle is called {\em Hamiltonian}. Similarly, a {\em Hamiltonian path} of a graph $G$ is a spanning path and a graph containing a Hamiltonian path is called {\em traceable}.
A (vertex) {\em  cutset} $S$ of a graph $G$ is a subset of its vertices for which $G-S$ is disconnected, where $G-S$ is the graph formed from $G$ by removing the vertices in $S$ and the edges incident to at least one vertex in $S$.
A graph is {\em $t$-tough} if for every cutset $S$, $c(G-S)\leq |S|/t$, where $c(H)$ is the number of connected components of a graph $H$. 
Chv\'atal~\cite{Ch73} conjectured that $t$-toughness implies Hamiltonicity for a sufficiently large $t$. 
Here we mainly focus on $1$-tough graphs which we just call tough. Note that every tough graph $G$ must be $2$-connected, since $G$ cannot contain a cut vertex.

Here we consider the Hamiltonicity, traceability and toughness of coline graphs. Note that for any graph $G$, $\co(G\cup K_1) = \co(G)$. So, without loss of generality, we only consider graphs with no trivial components.
We let $K_3^+$ be the graph obtained from $K_3$ by adding one new vertex that is only adjacent to one vertex in $K_3$. We also let $K_3 \circ K_1$ be the graph
obtained from $K_3$ by adding for every vertex $v$ of $K_3$ some new vertex $w$ only 
adjacent to $v$.
Chartrand,~Hevia,~Jarrett and~Schultz~\cite{ChHeJaSc97} gave several sufficient conditions that guarantee a coline graph is Hamiltonian and made two conjectures about weaker sufficient conditions that still guarantee Hamiltonicity of coline graphs.
Wu and Meng~\cite{WuMe04} proved the following, thus confirming one of Chartrand et al~\cite{ChHeJaSc97} conjectures and disproving the other.
\begin{thm}\label{t:WuMe}
Let $G$ be a graph with $m$ edges and maximum degree $\Delta$. Then $\co(G)$ is Hamiltonian unless at least one of the following is satisfied
\begin{itemize}
\item[(i)] $m < 2\Delta$;
\item[(ii)] $m=2\Delta$ and two vertices of degree $\Delta$ are adjacent in $G$;
\item[(iii)] $G$ is isomorphic to one of $K_3\cup P_3$, $K_3 \cup 2K_2$ and $C_4\cup K_2$;
\item[(iv)] $G$ has a subgraph isomorphic to\\
$
\begin{cases}
K_3^+ &\text{ if }m=6 \\
K_4^- \text{ or } K_3\circ K_1 &\text{ if }m=7 \\
K_4 &\text{ if }m=8 
\end{cases}
$
\item[(v)] $G$ is isomorphic to $K_5$,
\end{itemize}
\end{thm}
\noindent
where $P_n$ is the path on $n$ vertices and for any graphs $H$ and $H'$ and integer $k\geq 2$, $H \cup H'$ denotes the graph consisting exactly of vertex disjoint copies of $H$ and $H'$ and $kH = \cup_{i=1}^k H $.

Liu~\cite{Li05} quotes a result in an unpublished paper of theirs which also characterised Hamiltonian coline graphs; see Theorem 1 and Figure 1 in \cite{Li05}.
From this one can show that there are in fact 21 graphs that satisfy (iii)--(iv)  but neither (i) or (ii) of \cref{t:WuMe} namely the three graphs in Figure~\ref{f:graphH} and the 18 graphs in Figure~\ref{f:exceptions}. Here we will provide a proof of this independent from \cref{t:WuMe} and Liu's paper; see \cref{c:exceptions}.
Liu~\cite{Li05} also characterises graphs whose coline graphs are {\em pancyclic} and {\em bi-pancyclic}, which are graphs with cycles of all possible lengths and bipartite graphs with cycles of all possible even lengths, respectively.

One of the main goals of this paper is to give a more straightforward proof of \cref{t:WuMe} using the longest cycle method.
We in fact prove the following reformulation of \cref{t:WuMe}, where $H_1$, $H_2$ and $H_3$ are the graphs depicted in Figure~\ref{f:graphH}.
\begin{thm}\label{t:main}
Let $G$ be a graph and $L$ be its coline graph. If $L$ is tough then $L$ is Hamiltonian unless $G$ is isomorphic to $K_5$, $H_1$, $H_2$ or $H_3$.
\end{thm}
\begin{center}
\begin{figure}
  \centering
\begin{minipage}{\textwidth}
    \centering
    \begin{subfigure}[b]{0.15\textwidth}
\begin{tikzpicture}[scale=\hscale]

 \draw[line width = \lwid, color = \mcon]{
(-\s,-\s)-- (0,0.73*\s)
};
 \draw[line width = \lwid, color = \mctw]{
(-\s,-\s)-- (\s,-\s)
};
\draw[line width = \lwid, color = \mcth]{
(\s,-\s)-- (0,0.73*\s)
};
\draw[line width = \lwid, color = \mcfo]{
(-\s,-\s)--(-2*\s,0.73*\s)
};
\draw[line width = \lwid, color = \mcfi]{
(0,0.73*\s)--(2*\s,0.73*\s)
};
\draw[line width = \lwid, color = \mcsi]{
(\s,-\s)--(0,-2.73*\s)
};
\draw[line width = \lwid, color = \mcse]{
 (2*\s,0.73*\s)  edge[bend left=45] (0,-2.73*\s)
};
\draw[line width = \cwid, color = black!50]{ 
(-\s,-\s)  node[circle, draw, fill=black!10,inner sep=\inse, minimum width=\inci] {}
(\s,-\s) node[circle, draw, fill=black!10,inner sep=\inse, minimum width=\inci] {}
(0,0.73*\s) node[circle, draw, fill=black!10,inner sep=\inse, minimum width=\inci] {}
(-\s,-\s)  node[circle, draw, fill=black!10,inner sep=\inse, minimum width=\inci] {}
(0,-2.73*\s) node[circle, draw, fill=black!10,inner sep=\inse, minimum width=\inci] {}
(-2*\s,0.73*\s) node[circle, draw, fill=black!10,inner sep=\inse, minimum width=\inci] {}
(2*\s,.73*\s) node[circle, draw, fill=black!10,inner sep=\inse, minimum width=\inci] {}
};
\end{tikzpicture}
\end{subfigure}
\hspace{30mm}
\begin{subfigure}[b]{0.15\textwidth}
\begin{tikzpicture}[scale=\hscale]

 \draw[line width = \lwid, color = \mcon]{
(-\s,-\s)-- (0,0.73*\s)
};
 \draw[line width = \lwid, color = \mctw]{
(-\s,-\s)-- (\s,-\s)
};
\draw[line width = \lwid, color = \mcth]{
(\s,-\s)-- (0,0.73*\s)
};
\draw[line width = \lwid, color = \mcfo]{
(-\s,-\s)--(-2*\s,0.73*\s)
};
\draw[line width = \lwid, color = \mcfi]{
(0,0.73*\s)--(2*\s,0.73*\s)
};
\draw[line width = \lwid, color = \mcsi]{
(\s,-\s)--(0,-2.73*\s)
};
\draw[line width = \lwid, color = \mcse]{
(2*\s,.73*\s)--(3*\s,-\s)
};

\draw[line width = \cwid, color = black!50]{ 
(-\s,-\s)  node[circle, draw, fill=black!10,inner sep=\inse, minimum width=\inci] {}
(\s,-\s) node[circle, draw, fill=black!10,inner sep=\inse, minimum width=\inci] {}
(0,0.73*\s) node[circle, draw, fill=black!10,inner sep=\inse, minimum width=\inci] {}
(-\s,-\s)  node[circle, draw, fill=black!10,inner sep=\inse, minimum width=\inci] {}
(3*\s,-\s)  node[circle, draw, fill=black!10,inner sep=\sf*0.125, minimum width=\inci] {}
(0,-2.73*\s) node[circle, draw, fill=black!10,inner sep=\inse, minimum width=\inci] {}
(-2*\s,0.73*\s) node[circle, draw, fill=black!10,inner sep=\inse, minimum width=\inci] {}
(2*\s,.73*\s) node[circle, draw, fill=black!10,inner sep=\inse, minimum width=\inci] {}
};
\end{tikzpicture}
\end{subfigure}
\hspace{30mm}
\begin{subfigure}[b]{0.15\textwidth}
\begin{tikzpicture}[scale=\hscale]
 \draw[line width = \lwid, color = \mcon]{
(-\s,-\s)-- (0,0.73*\s)
};
 \draw[line width = \lwid, color = \mctw]{
(-\s,-\s)-- (\s,-\s)
};
\draw[line width = \lwid, color = \mcth]{
(\s,-\s)-- (0,0.73*\s)
};
\draw[line width = \lwid, color = \mcfo]{
(-\s,-\s)--(-2*\s,0.73*\s)
};
\draw[line width = \lwid, color = \mcfi]{
(0,0.73*\s)--(2*\s,0.73*\s)
};
\draw[line width = \lwid, color = \mcsi]{
(\s,-\s)--(0,-2.73*\s)
};
\draw[line width = \lwid, color = \mcse]{
(3*\s,-\s)--(2*\s,-2.73*\s)
};  
 \draw[line width = \lwid, color = blue]{
};
\draw[line width = \cwid, color = black!50]{ 
(-\s,-\s)  node[circle, draw, fill=black!10,inner sep=\inse, minimum width=\inci] {}
(\s,-\s) node[circle, draw, fill=black!10,inner sep=\inse, minimum width=\inci] {}
(0,0.73*\s) node[circle, draw, fill=black!10,inner sep=\inse, minimum width=\inci] {}
(-\s,-\s)  node[circle, draw, fill=black!10,inner sep=\inse, minimum width=\inci] {}
(3*\s,-\s)  node[circle, draw, fill=black!10,inner sep=\sf*0.125, minimum width=\inci] {}
(0,-2.73*\s) node[circle, draw, fill=black!10,inner sep=\inse, minimum width=\inci] {}
(-2*\s,0.73*\s) node[circle, draw, fill=black!10,inner sep=\inse, minimum width=\inci] {}
(2*\s,.73*\s) node[circle, draw, fill=black!10,inner sep=\inse, minimum width=\inci] {}
(2*\s,-2.73*\s)node[circle, draw, fill=black!10,inner sep=\inse, minimum width=\inci] {}
};
\end{tikzpicture}
\end{subfigure}
\hfill
\caption{The graphs $H_1$, $H_2$, $H_3$}
\label{f:graphH}
  \end{minipage}
\begin{minipage}{\textwidth}
    \centering
\begin{subfigure}[b]{0.15\textwidth}
\begin{tikzpicture}[scale=\hlscale]
 \draw[line width = \lwid, color = black]{
(0,0)--(90+1*\a:\s)
(0,0)--(90+3*\a:\s)
(0,0)--(90-1*\a:\s)
(0,0)--(90-2*\a:\s)
(90+0*\a:\s)--(90+1*\a:\s)
(90+2*\a:\s)--(90+3*\a:\s)
(90-2*\a:\s)--(90-1*\a:\s)
(90+0*\a:\s)--(90+2*\a:\s)
(90+0*\a:\s)--(90-2*\a:\s)
(90+2*\a:\s)--(90-2*\a:\s)
};
\draw[line width = \cwid, color = \mcse]{ 
(0,0)node[circle, draw, fill=\mcse!60,inner sep=\inse, minimum width=\inci] {}
};
\draw[line width = \cwid, color = \mcsi]{ 
(90:\s)node[circle, draw, fill=\mcsi!60,inner sep=\inse, minimum width=\inci] {}
};
\draw[line width = \cwid, color = \mcon]{ 
(90+\a:\s)node[circle, draw, fill=\mcon!60,inner sep=\inse, minimum width=\inci] {}
};
\draw[line width = \cwid, color = \mcfi]{ 
(90+2*\a:\s)node[circle, draw, fill=\mcfi!60,inner sep=\inse, minimum width=\inci] {}
};
\draw[line width = \cwid, color = \mctw]{ 
(90+3*\a:\s)node[circle, draw, fill=\mctw!60,inner sep=\inse, minimum width=\inci] {}
};
\draw[line width = \cwid, color = \mcth]{ 
(90-1*\a:\s)node[circle, draw, fill=\mcth!60,inner sep=\inse, minimum width=\inci] {}
};
\draw[line width = \cwid, color = \mcfo]{ 
(90-2*\a:\s)node[circle, draw, fill=\mcfo!60,inner sep=\inse, minimum width=\inci] {}
};
\end{tikzpicture}
\end{subfigure}
\hspace{30mm}
\begin{subfigure}[b]{0.15\textwidth}
\begin{tikzpicture}[scale=\hlscale]
 \draw[line width = \lwid, color = black!]{
(0,0)--(90+0*\a:\s)
(0,0)--(90+1*\a:\s)
(0,0)--(90+3*\a:\s)
(0,0)--(90-1*\a:\s)
(0,0)--(90-2*\a:\s)
(90+0*\a:\s)--(90+1*\a:\s)
(90+2*\a:\s)--(90+3*\a:\s)
(90-2*\a:\s)--(90-1*\a:\s)
(90+0*\a:\s)--(90+2*\a:\s)
(90+0*\a:\s)--(90-2*\a:\s)
(90+2*\a:\s)--(90-2*\a:\s)
};
\draw[line width = \cwid, color = \mcse]{ 
(0,0)node[circle, draw, fill=\mcse!60,inner sep=\inse, minimum width=\inci] {}
};
\draw[line width = \cwid, color = \mcsi]{ 
(90:\s)node[circle, draw, fill=\mcsi!60,inner sep=\inse, minimum width=\inci] {}
};
\draw[line width = \cwid, color = \mcon]{ 
(90+\a:\s)node[circle, draw, fill=\mcon!60,inner sep=\inse, minimum width=\inci] {}
};
\draw[line width = \cwid, color = \mcfi]{ 
(90+2*\a:\s)node[circle, draw, fill=\mcfi!60,inner sep=\inse, minimum width=\inci] {}
};
\draw[line width = \cwid, color = \mctw]{ 
(90+3*\a:\s)node[circle, draw, fill=\mctw!60,inner sep=\inse, minimum width=\inci] {}
};
\draw[line width = \cwid, color = \mcth]{ 
(90-1*\a:\s)node[circle, draw, fill=\mcth!60,inner sep=\inse, minimum width=\inci] {}
};
\draw[line width = \cwid, color = \mcfo]{ 
(90-2*\a:\s)node[circle, draw, fill=\mcfo!60,inner sep=\inse, minimum width=\inci] {}
};
\end{tikzpicture}
\end{subfigure}
\hspace{30mm}
\begin{subfigure}[b]{0.15\textwidth}
\begin{tikzpicture}[scale=\hlscale]
  \pgfmathsetmacro{\a}{60}
 \draw[line width = \lwid, color = black]{
(0,0)--(90+0*\a:\s)
(0,0)--(90+1*\a:\s)
(0,0)--(90+2*\a:\s)
(0,0)--(90+3*\a:\s)
(0,0)--(90-1*\a:\s)
(0,0)--(90-2*\a:\s)
(90+0*\a:\s)--(90+1*\a:\s)
(90+2*\a:\s)--(90+3*\a:\s)
(90-2*\a:\s)--(90-1*\a:\s)
(90+0*\a:\s)--(90+2*\a:\s)
(90+0*\a:\s)--(90-2*\a:\s)
(90+2*\a:\s)--(90-2*\a:\s)
};
\draw[line width = \cwid, color = \mcse]{ 
(0,0)node[circle, draw, fill=\mcse!60,inner sep=\inse, minimum width=\inci] {}
};
\draw[line width = \cwid, color = \mcsi]{ 
(90:\s)node[circle, draw, fill=\mcsi!60,inner sep=\inse, minimum width=\inci] {}
};
\draw[line width = \cwid, color = \mcon]{ 
(90+\a:\s)node[circle, draw, fill=\mcon!60,inner sep=\inse, minimum width=\inci] {}
};
\draw[line width = \cwid, color = \mcfi]{ 
(90+2*\a:\s)node[circle, draw, fill=\mcfi!60,inner sep=\inse, minimum width=\inci] {}
};
\draw[line width = \cwid, color = \mctw]{ 
(90+3*\a:\s)node[circle, draw, fill=\mctw!60,inner sep=\inse, minimum width=\inci] {}
};
\draw[line width = \cwid, color = \mcth]{ 
(90-1*\a:\s)node[circle, draw, fill=\mcth!60,inner sep=\inse, minimum width=\inci] {}
};
\draw[line width = \cwid, color = \mcfo]{ 
(90-2*\a:\s)node[circle, draw, fill=\mcfo!60,inner sep=\inse, minimum width=\inci] {}
};
\end{tikzpicture}
\end{subfigure}
\end{minipage}
\caption{The coline graphs of $H_1$, $H_2$ and $H_3$}
\label{f:colineH}
\end{figure}
\end{center}
\vspace{-10mm}
Line graphs and to an extent coline graphs have been well studied in their own right; see for example \cite{BeBa21}.
It is therefore natural to consider Hamiltonicity of line graphs and coline graphs.
Here we give two other motivations for studying Hamiltonicity of coline graphs.
The $k$-th power of a graph $G$ is the graph $G^k$ whose vertex set is $V(G)$ for which two vertices $v$ and $w$ are adjacent if and only if their distance in $G$ is at most $k$. 
The main motivation of the author for studying the Hamiltonicity of coline graphs is its relation to the so called {\em cyclic matching sequenceability} of a graph $G$, denoted $\cms(G)$, a positive integer-valued graph parameter;
see \cite{BrKiMeSc12} or \cite{HoMa21}  for a precise definition of cyclic matching sequenceability and see \cite{HoMa21}  for a brief overview of previous results on cyclic matching sequenceability. 
It can be shown that $\cms(G) = k$ if and only if $\co(G)$ contains the $(k-1)$-th power of a Hamiltonian cycle;
Chiba and Nakano~\cite{ChNa16} first observed a similar fact for a natural non-cyclic variant of cyclic matching sequenceability. Given that $H^1 =H$ for all graphs $H$, determining for what graphs $G$ is  $\co(G)$ is Hamiltonian is equivalent to determining for what graphs $G$ is $\cms(G)\geq 2$, which is of interest since 2 is the smallest non-trivial value the cyclic matching sequenceability of a graph can be because $\cms(G) \geq 1$ is true for all graphs $G$.

For a set of graphs $\fH$, a graph $G$ is $\fH$-free if it is does not contain an induced subgraph isomorphic to $H$ for all $H \in \fH$. For brevity, we write $H$-free rather than $\{H\}$-free.
The Hamiltonicity of $\fH$-free graphs and $H$-free graphs has been considered for various sets of graphs $\fH$ and graphs $H$. We briefly go through several graphs $H$ most relevant to us here.
It has been shown \cite{Be70} (see also \cite{BeBa21} and \cite{RyVr11} ) that a graph $G$ is the line graph of some graph if and only if it is $\fS$-free for a particular set $\fS$ of 9 graphs. It follows that a graph is a coline graph of some graph if and only if it is $\overline{\fS}$-free, where $\overline{\fS}$ is the set of the complements of the graphs in $\fS$. Two of the graphs in $\overline{\fS}$ are $K_2 \cup 3K_1$, the graph with a single edge and 3 isolated vertices, and $P_2 \cup P_3$, the graph consisting of two vertex disjoint paths on 2 and 3 vertices, respectively.
Shan~\cite{Sh21} proved that every $15$-tough $(P_2 \cup P_3)$-free is Hamiltonian, which appears to be the strongest result on Hamiltonicity $(P_2 \cup P_3)$-free graphs to date. 
Xu, Li and Zhou~\cite{XuLiZh24} proved that any $\max\{2k-2,2\}$-connected $(K_2\cup kK_1)$-free graph is Hamiltonian, generalising results in \cite{HaEl} and \cite{ShSh22}, where a graph $G$ is $k$-connected if the minimum size of a vertex cutset of $G$ is at least $k$. More recently, Ota and Sanka~\cite{OtSa24} improved Xu, Li and Zhou's result by proving the following.
\begin{thm}[\cite{OtSa24}]\label{t:OtSa24}
Let $k \geq 2$ and $G$ be a tough $k$-connected $(K_2\cup kK_1)$-free graph with minimum degree at least $\frac{3(k-1)}{2}$. Then $G$ is Hamiltonian or is the Petersen graph.
\end{thm}

\noindent
Though coline graphs are $(K_2\cup 3K_1)$-free,  applying \cref{t:OtSa24} for $k= 3$ yields \cref{t:main} only in the case that $\co(G)$ is 3-connected (and has minimum degree 3) and we will see that in fact the most difficult case of \cref{t:main} to prove is for graphs $G$ for which $\co(G)$ (is tough and) has minimum degree 2 (and so is 2-connected but not 3-connected). 
Therefore, we do not utilise \cref{t:OtSa24} in the proof of \cref{t:main}.

Apart from providing an alternate proof of \cref{t:WuMe} and Liu's~\cite{Li05} result, here we characterise tough coline graphs and traceable coline graphs. Using \cref{t:main}, 
we provide a description of non-tough coline graphs; see \cref{c:exceptions}. 
For a graph $L$, a vertex $v$ in $L$ is \emph{dominating} if it is adjacent to every other vertex in $L$.
For a graph $L$, let $L^*$ be the graph obtained from $L$ by adding a new dominating vertex.
It is a simple exercise to prove that $L$ is traceable if and only if $L^*$ is Hamiltonian. This motivates the following definition. A graph $L$ is \emph{pseudo-tough} if $L^*$ is tough. 
Notice also that if $L$ is in fact the coline graph of a graph $G$, then $L^*$ is the coline graph of the graph $G \cup K_2$. Note however $G \cup K_2$ is not necessarily a unique coline root graph of $L^*$; see \cref{t:Whit}.
Using this we also prove the following analogue of \cref{t:main} for traceability.
\begin{cor}\label{c:main}
Let $G$ be a graph whose coline graph $L=\co(G)$ is pseudo-tough. Then $L$ is traceable unless $G$ is $ K_3\circ K_1$. 
\end{cor}
\noindent
We also give in \cref{c:Trace} a more explicit description of which coline graphs are not traceable.
A lovely coincidence is that $\co(K_3\circ K_1)=K_3\circ K_1$ and apart from $C_5$ there is no other graph $G$ such that $G=\co(G)$; see \cite{Ai69}. In fact $C_5$ is one of the few graphs whose coline graph is Hamiltonian but isn't pancyclic~\cite{Li05}.

The paper is organised as follows. \cref{s:pre} has preliminary definitions and results. In \cref{s:main} we give the proof of \cref{t:main}. In the final section, \cref{s:TNT}, we provide a characterisation of graphs whose coline graphs are not tough in \cref{c:exceptions}. The section and the paper conclude with the proof of \cref{c:main}, as well with \cref{c:Trace}, which provides a more explicit characterisation of graphs whose coline graphs are not traceable.

\section{Preliminaries}\label{s:pre}
For this paper, we let $G$ be a graph with maximum degree $\Delta$.
We will use the labels of edges from $G$ to denote the corresponding vertices in $\co(G)$; we will also make clear whether we are considering those elements as edges of $G$ or as vertices of $\co(G)$. 
We say that a graph $G$ is of \emph{type-$I$} if $G$ has an edge $e$ that is adjacent to every other edge and $\co(G)$ has exactly 2 components. 
For $k\geq 2$ we let $F_k$ denote the graph formed from $K_{1,k}$ by adding one edge incident to two of the vertices with degree 1. Note that $F_2$ is in fact $K_3$ and $F_3$ is $K_3^+$ as defined in the Introduction.
We use $K_4^-$ to denote the graph formed by removing an edge from $K_4$.
For a graph $G'$ such that $\co(G')$ is disconnected, let $\rho(G') = |E(G')|+c(\co(G'))$ and note that any supergraph $G$ of $G'$ with less than $\rho(G')$ edges cannot have a tough coline graph, since the edges of $S=E(G)-E(G')$ (considered as vertices in $\co(G)$) would be a cutset of $\co(G)$ such that $c(\co(G)-S)=c(\co(G'))>|S|$.
Chartrand et al~\cite{ChHeJaSc97} and Wu and Meng et al.~\cite{WuMe04} categorised which graphs $G$ have a connected coline graph.
Here we give a more explicit characterisation of graphs $G$ for which $\co(G)$ is disconnected along with the values of $c(\co(G))$ and $\rho(G)$,
with the following lemma.
\begin{lem}\label{l:disconColineGraph}
If $G$ is a graph and $L=\co(G)$ is disconnected then $G$ is either
\begin{enumerate}
\item[(i)] $K_{1,\lambda}$ and $c(L)=\lambda$ and $\rho(G)=2\lambda$, for some $\lambda \geq 1$; 
\item[(ii)] of type-$I$ and $c(L)=2$ and $\rho(G)=|E(G)|+2$; 
\item[(iii)] $C_4$ and $c(L)=2$ and $\rho(G)=6$;
\item[(iv)] $F_k$ and $c(L)=3$ and $\rho(G) = k+4$ for some, $k\geq 2$;  
\item[(v)] $K_4^-$ and $c(L) = 3$ and $\rho(G)=8$; 
\item[(vi)] $K_4$ and $c(L)=3$ and $\rho(G)=9$.
\end{enumerate}
\end{lem}
\begin{proof} 
We separate the proof into cases based on the number of non-trivial components of $L$.
First suppose that two components $H$ and $H'$ of $L$ are non-trivial. 
Let $x_1$ and $y_1$ be two adjacent vertices in $H$ and 
$x_2$ and $y_2$ be two adjacent vertices in $H'$.
Then as edges in $G$, $x_1$ and $y_1$ are non-adjacent as are $x_2$ and $y_2$.
On the other hand $x_1$ and $y_1$ are in a different component of $L$ to $x_2$ and $y_2$ and consequently each of the edges $x_1$ and $y_1$ is adjacent to both $x_2$ and $y_2$.
Since $x_1,y_1$ and $x_2,y_2$ are each disjoint pairs of 
non-adjacent edges and every edge from the first pair is adjacent to every edge in the second pair, these four edges form a $4$-cycle; 
let $V$ be the set of vertices of this $4$-cycle. 
If there were an edge $e$ in $G$ incident to at most one vertex in $V$, then $e$ would not be adjacent to at least one $x_1$ and $y_1$ and at least one of $x_2$ and $y_2$. 
It would then follow that the vertex $e$ is in the same connected component of $L$ as one of  $x_1$ and $y_1$ and one of  $x_2$ and $y_2$, contradicting the fact that  $x_1$ and $y_1$ are in a different component to $x_2$ and $y_2$.
Thus $V(G)=V$ and so $G$ is a supergraph of $C_4$ and a subgraph of $K_4$. It follows that $G$ is isomorphic to either $C_4$, $K_4^-$ or $K_4$ and $L$ has $2$, $3$ and $3$ components, respectively and therefore $\rho(C_4)=6$, $\rho(K_4^-)=8$ and $\rho(K_4)=9$.

Now suppose that only one component $H$ of $L$ is non-trivial and let $x$ and $y$ be adjacent vertices in $H$. 
Similarly to the previous paragraph, $x$ and $y$ are non-adjacent as edges in $G$.
Let $V$ be the set of vertices in $G$ incident to either $x$ or $y$.
As $L$ is disconnected, there is another component $H'$ of $L$ containing a vertex $z$ and, as every component other than $H$ is trivial, $z$ is an isolated vertex. 
Therefore, $z$ as an edge is adjacent to every other edge in $G$ and, in particular, must be incident to an endpoint $v$ of $x$ and an endpoint $w$ of $y$. 
Then every edge of $G$ other than $z$ must be incident to exactly one of $v$ and $w$. 

If $c(L)=2$, and hence $\rho(G)=|E(G)|+2$, then $G$ is of type-$I$, with $z$ being the edge in $G$ adjacent to every other edge. 
Otherwise there exists a vertex $z'$ in a different and necessarily trivial component of $L$ to $x,y$ and $z$. 
So, the edge  $z'$ is adjacent to every other edge in $G$. 
In particular, it must be adjacent to $x,y$ and $z$, which is possible only if $z'$ is incident $v$ or is incident to $w$. 
Without loss of generality suppose that $z'$ is incident to $v$. 
Then as $z$ and $z'$ are adjacent to every other edge in $G$, every edge other than $y$ must be incident to $v$. 
Since $y,z$ and $z'$ form a triangle, it follows that $G$ is isomorphic to $F_k$ for some $k \geq 3$.
Moreover, every edge of $G$ other than $z$ and $z'$ must, as vertices, be adjacent to $y$ and so must be in the component $H$ of $L$. 
So $c(\co(F_k))=3$ and $\rho(G)= |E(G)|+c(\co(G))= k+1+3=k+4$.

Finally if every component of $L$ is trivial, then every edge in $G$ is adjacent to every other edge. Thus $G$ is either $K_3$ or $K_{1,\lambda}$ for some $\lambda\geq 1$ and clearly $L$ has $3$ and $\lambda$ components, respectively and therefore $\rho(K_3)=6$ and $\rho(K_{1,\lambda})=2\lambda$. This completes the proof, because $K_3 = F_2$.
\end{proof}
\noindent 
We call a graph $G$ \emph{exceptional} if it is one of the graphs in \cref{l:disconColineGraph}.

The proof of \cref{t:main} will use the longest cycle method where, for the sake of contradiction we suppose our coline graph $L$ does not have a Hamiltonian cycle. So 
instead $L$ must have a longest non-Hamiltonian cycle $C$; $L$ indeed has a cycle, 
since we only consider tough coline graphs, which must be 2-connected, yet any graph with no cycle is a forest which is not $2$-connected.
We will need the following notation and properties related to longest cycles of graphs.
For a cycle $C$, with some fixed orientation, and vertices $x,y \in V(C)$, let $x \ora y$ be the path in $C$ from $x$ to $y$ which follows the orientation of $C$ and $x \ola y$ be the path from $x$ to $y$ in $C$ opposing the orientation of $C$.
For a vertex $x \in V(C)$, let $x^{+i}$ and $x^{-i}$ be the vertices that occur $i$ positions after $x$ and $i$ position before $x$ with respect to the orientation of $C$, respectively. We simply write $x^+$ and $x^-$ for $x^{+1}$ and $x^{-1}$, respectively. For a subset $S \subseteq V(C)$, we let $S^+ = \{x^+\;:\; x \in S\}$, $S^- = \{x^-\;:\; x \in S\}$ and $S^\pm = S^+\cup S^-$.
For two vertices $u$ and $v$, a $uv$-path is a path from $u$ to $v$.
The {\em neighbourhood} of a vertex $v$ in a graph $L$, denoted $N_L(v)$, is the set vertices in $L$ adjacent to $v$. We let $N_L(H)=\bigcup_{v \in V(H)} \left(N_L(v)- V(H)\right) $ for any subgraph $H$ of a graph $L$.
We will need the following properties of longest cycles in graphs from \cite{OtSa24}.
\begin{lem}\label{l:propsOfLongestCycle}
Let $L$ be a graph, $C$ be a longest cycle in $L$ and $H$ be a connected component of $L-V(C)$. Then the following hold.
\begin{itemize}
\item[(i)] If $x \in N_L(H)$ then $x^+,x^- \notin N_L(H)$.
\item[(ii)] If $x,y \in N_L(H)$, then there is no $x^+y^+$-path or  $x^-y^-$-path with all internal vertices in $L-V(C)$.
\item[(iii)] $N_L(H)^+\cup\{x\}$ and $N_L(H)^-\cup\{x\}$ are independent sets in $L$ for all $x \in V(H)$.
\end{itemize}
\end{lem}
\begin{proof}
Note that as $H$ is a connected component of $L-V(C)$, $N_L(H) \subseteq V(C)$.
For (i) suppose otherwise that $x^\epsilon$  is in $N_L(H)$ for $\epsilon \in \{+,-\}$. Let $P^\epsilon$ be a path from $x$ to $x^\epsilon$ with all internal vertices in $H$.
Then we have the contradiction that 
\[
C'=
\begin{cases} 
xP^+x^+ \overrightarrow{C}x &\text{ if $x^+ \in N_L(H)$}\\
x P^- x^- \ola x &\text{ if $x^- \in N_L(H)$}\\
\end{cases}
\]
is a cycle longer than $C$.

For (ii), suppose for a contradiction that $P$ is either a $x^+y^+$-path or a $x^-y^-$-path, with all internal vertices in $L-V(C)$.
By (i), $x^+,x^-,y^+$ and $y^-$ are all not in $N_L(H)$ and so no vertex of $P$ can be in $H$.
Let $P_{xy}$ be a $xy$-path with all internal vertices in $H$ and note that $P$ and $P_{xy}$ are necessarily vertex disjoint.
Then we have the contradiction
\[
C' = 
\begin{cases}
x P_{xy}y \ola x^+ P y^+  \ora x &\text{ if $P$ is a $x^+y^+$-path} \\
  xP_{xy} y \ora x^- P y^-  \ola x&\text{ if $P$ is a $x^-y^-$-path} \\
\end{cases}
\]
is a cycle longer than $C$.

Finally by (ii) $N_L(H)^+$ and $N_L(H)^-$ are each independent sets and $x$ is not adjacent to any vertex in $N_L(H)^\pm$ by (i). Thus (iii) holds.
\end{proof}
\noindent
We will implicitly use \cref{l:propsOfLongestCycle}(i) throughout the remainder of the paper.
While (i) establishes that $N_L(H)\cap N_L(H)^+ = \emptyset$ and $N_L(H)\cap N_L(H)^- = \emptyset$ for any connected component $H$ in $L-V(C)$, it is possible that $N_L(H)^+ \cap N_L(H)^- \neq \emptyset$.

For the following two results, we let $L$ be a graph, $C$ be a longest cycle $C$ in $L$ with a fixed orientation, $x \in V(L)-V(C)$ and $x_1,\ldots, x_d$ be the neighbours of $x$ in order as they appear in $C$.
Also we write $x_i \prec_C x_j\preceq_C x_k$ if $x_j$ occurs between $x_i$ and $x_k$ with respect to the orientation of $C$, where $x_i \neq x_j$ but possibly $x_j=x_k$.
We also require several other properties of longest cycles.
\begin{lem}\label{l:common_arg}
For any $x_i\prec_C x_j \preceq_C x_k$ (with $x_i\neq x_k$), 
either
$\{x_i^+ x_j, x_i^- x_j, x_i^+ x_k, x_i^-x_k\}\cap E(L) = \emptyset$ or $x_j^- x_k^+ \notin E(L)$. 
In particular, either  $\{x_i^+ x_j$, $x_i^- x_j\}\cap E(L) = \emptyset$ or $x_j^- x_j^+ \notin E(L)$ for all distinct $i,j \in [d]$.
\end{lem} 
\begin{proof}
If  $x_j^- x_k^+$ and at least one of $x_i^+ x_j$, $x_i^- x_j$, $x_i^+ x_k$ or $x_i^- x_k$ were an edge in $L$, then 
\[
C' = 
\begin{cases}
x x_i \overrightarrow{C} x_j^- x_k^+ \overrightarrow{C}x_i^- x_j \overrightarrow{C} x_k x &\text{if }x_i^- x_j \in E(L)\\
x x_i \overrightarrow{C} x_j^- x_k^+ \overrightarrow{C}x_i^- x_k \ola x_j x &\text{if }x_i^- x_k \in E(L)\\
x x_i \overleftarrow{C} x_k^+ x_j^- \overleftarrow{C} x_i^+ x_j \overrightarrow{C} x_k x &\text{if }x_i^+ x_j \in E(L)\\
x x_i \overleftarrow{C} x_k^+ x_j^- \overleftarrow{C} x_i^+ x_k \overleftarrow{C} x_j x &\text{if }x_i^+ x_k \in E(L)
\end{cases}
\]
would be  a cycle longer than $C$. The second assertion of the lemma is obtained by setting $j=k$.
\end{proof}

\begin{lem}\label{l:nonXYedgesExtendCycle}
Suppose that $ww^+ \in x_i\overrightarrow{C}x_j$ with $i\neq j$. Then at least one of $wz$ and $w^+z'$ is not an edge in $L$ where $\{z,z'\}= \{x_i^-,x_j^+\}$.
\end{lem}
\begin{proof}
If this were not the case, then 
\[
C' = 
\begin{cases}
x x_i \overrightarrow{C} w x_i^- \overleftarrow{C} x_j^+ w^+ \overrightarrow{C} x_j x &\text{if } (z,z')=(x_i^-,x_j^+) \\
x x_i \overrightarrow{C} w x_j^+ \overrightarrow{C} x_i^- w^+ \overrightarrow{C} x_j x &\text{if } (z,z')=(x_j^+,x_i^-)
\end{cases}
\]
would be a cycle longer than $C$.
\end{proof}
For clarity, when applying any of the above lemmas, we specify (in order) the pair $(i,j)$ (or the triple $(i,j,k)$ in the case of \cref{l:common_arg}) which it is being applied to.

\section{Proof of \cref{t:main}}\label{s:main}
In this section we give the proof of \cref{t:main}. We first prove that $H_1,H_2$ and $H_3$ are genuine exceptions to the theorem. That is, we show that $H_1,H_2$ and $H_3$ are tough but not Hamiltonian. Next we prove a key lemma that shows that for any longest cycle $C$ in a non-Hamiltonian tough coline graph $L$, the connected components of $L-V(C)$ are trivial.
Finally we give the proof of \cref{t:main}, which proceed according to the structure of $N_L(x)^+ \cup \{x\}$ and $N_L(x)^- \cup \{x\}$, where $x \in L-V(C)$ is an isolated vertex in $L$. 

We begin by proving that the coline graphs of $H_1,H_2$ and $H_3$ are tough but not Hamiltonian.
\begin{lem}\label{l:Hs}
The graphs $H_1,H_2$ and $H_3$ all have a coline graph that is tough but is not Hamiltonian.
\end{lem}
\begin{proof}
Let $L_1,L_2$ and $L_3$ be the  coline graphs of $H_1,H_2$ and $H_3$, respectively; see Figures~\ref{f:graphH} and \ref{f:colineH}.
First we show that $L_1,L_2$ and $L_3$ are not Hamiltonian. 
Since $L_1$ and $L_2$ are spanning subgraphs of $L_3$, it suffices to show that $L_3$ is not Hamiltonian. Let $e_1,e_2$ and $e_3$ be the edges of the triangle of $H_3$ and $e$ the isolated edge. Then as vertices in $L_3$, $e_1,e_2$ and $e_3$ have degree 2 and are each adjacent to $e$. However this means that if there were a Hamiltonian cycle in $L_3$, then it must contain the 6 edges incident to $e_1$, $e_2$ or $e_3$, yet three of these edges are incident to $e$. Hence $L_3$ and so $L_1$ and $L_2$ are not Hamiltonian.

Finally we show that $L_1,L_2$ and $L_3$ are tough.
As $L_1$ is a spanning subgraph of $L_2$ and $L_3$, it suffices to prove that $L_1$ is tough.
Let $e_1,e_2$ and $e_3$ be the edges of the triangle of $H_1$ and $e$ the edge adjacent to none of $e_1,e_2,e_3$.
Suppose for a contradiction that $L_1$ is not tough. Then $L_1$ has a cutset $S$ 
for which $c(L_1-S) >|S|$. Note that $L_1-S = \co(H_1-S)$, since the edges of $H_1$ correspond to the vertices of $L_1$,
where $H_1-S$ is the graph formed from $H_1$ by removing the edges in $S$. Then as $L_1-S$ is disconnected, $H_1-S$  must be an exceptional graph. As $L_1$ is 2-connected, $S$ must have at least 2 elements and $c(L_1-S)\geq 3$.
Therefore by \cref{l:disconColineGraph}, $H_1-S$ must either be  $K_{3}, K_4^-$, $K_4$ or $F_k$ for some $k \geq 3$ and $|S|=2$ or $H_1-S$ is $K_{1,\lambda}$ for some $\lambda > |S| \geq 2$.
The latter is impossible for $\lambda \geq 4$ because every vertex in $H_1$ has degree at most 3 and for $\lambda=3$, because then $|S|=2$ and $H_1-S$ has 5 edges, but $K_{1,3}$ only has 3 edges. Of the former possibilities, only $K_4^-$ and $F_4$ contain exactly $|E(H_1)|-|S|=5$ edges, yet $H_1$ contains neither $K_4^-$ or $F_4$ as a subgraph. 
\end{proof}

The next result is the key structural lemma used in the proof of \cref{t:main}.

\begin{lem}\label{l:eachCompTrivial}
Let $G$ be a graph whose coline graph $L$ is tough and $C$ be a longest cycle of $L$. Then every connected component of $L-V(C)$ is trivial.
\end{lem}
\begin{proof}
Let $H$ be a connected component of $L-V(C)$ and suppose for the sake of contradiction that 
$H$ has two adjacent vertices $x$ and $y$. 
Then as edges in $G$, $x$ and $y$ are not adjacent. 
Let $N_L(H)=\{z_1,\ldots ,z_d\}$ be the neighbours of the vertices of $H$, which are all necessarily in $C$, where $d\geq 2$ because $L$ is 2-connected. By \cref{l:propsOfLongestCycle}(iii), $N_L(H)^+\cup \{x\}$ is an independent set in $L$.
So the elements of $N_L(H)^+$ are, as edges, all adjacent to each other and to $x$ and $y$. So the edges corresponding to the members of $N_L(H)^+$ are mutually adjacent and must be adjacent $x$ and $y$.
Since $N_L(H)^+$ is an independent set in $L$ by \cref{l:propsOfLongestCycle}(iii), the members of $N_L(H)^+$ must be mutually adjacent.
Moreover, because $x$ and $y$ are non-adjacent edges and $N_L(H)^+\cup \{x\}$ and $N_L(H)^+\cup \{y\}$ are independent sets by \cref{l:propsOfLongestCycle}(iii), each member of $N_L(H)^+$ is incident an endpoint of $x$ and an endpoint of $y$. As only at most two edges can simultaneously be adjacent and be incident to both an endpoint of $x$ and an endpoint of $y$, it follows that $d\leq 2$ and consequently $d=2$. Furthermore, 
the elements of $V(H)\cup N_L(H)^+$, as edges in $G$, must form an $F_{k}$, where $k=|V(H)|+1 \geq 3$, with the elements in $V(H)$ forming a $K_{1,k-2}\cup K_2$. If $|V(H)|\geq 3$, then no edge other than $z_1^+$ and $z_2^+$ can be adjacent to every element in $V(H)$ and so $z_1^+$ and $z_2^+$ are the only vertices in $L$ outside $N_L(H)\cup V(H)$. Consequently, $G$ is a supergraph of $F_k$ with $|N_L(H)|+|V(H)|+2= k+3 < \rho(F_k)$ edges, yet $L$ is tough, contrary to \cref{l:disconColineGraph}(iv).
So $|V(H)|=2$. As the elements of $\left(V(L)-V(H)\right)-\{z_1,z_2\}$ are, as edges in $G$, adjacent to every element in $V(H)=\{x,y\}$, the edges in $E(G)-\{z_1,z_2\}$ must form a graph $G'$ that is a supergraph of $F_3$ and a subgraph of $K_4$. But then $G$ is a supergraph of $G'$ with $|E(G')|+2 < \rho(G')$ edges, yet $L$ is tough contrary to \cref{l:disconColineGraph}(iv), (v) or (vi).
\end{proof}

\noindent We can now give the proof of \cref{t:main}.
\begin{proof}[Proof of \cref{t:main}]
By \cref{l:Hs} and the fact that $\co(K_5)$ is the Petersen graph, we can assume that $G$ is neither $K_5$, $H_1,H_2$ or $H_3$.
Let $L = \co(G)$ and  $C$ be a longest cycle in $L$ with some fixed orientation. Suppose for a contradiction that $C$ is not Hamiltonian.
By \cref{l:eachCompTrivial}, each connected component of $L-V(C)$ is trivial. 
So let $x$ be an isolated vertex in $L-V(C)$, with degree $d \geq 2$ in $L$, since $L$ is tough and so is 2-connected. Then every neighbour of $x$ in $L$ must be in $C$. So let $N_L(x) = \{x_1,x_2,\ldots, x_d\}$ be the set of neighbours of $x$ in $C$, in order with respect to the orientation of $C$, where 
we take the subscripts of the members of $N_L(x)$ modulo $d$.
Let $v$ be the common endpoint of $x_1^+$ and $x$ and $v'$ be the other endpoint of $x$.

We often interpret the conclusions of Lemmas~\ref{l:propsOfLongestCycle}--\ref{l:nonXYedgesExtendCycle} in terms of edges in $G$, rather than in terms of the vertices in $L$.
We will also make constant use of the following facts throughout the proof. 
By \cref{l:propsOfLongestCycle}(iii), $N_L(x)^+\cup \{x\}$ and $N_L(x)^-\cup \{x\}$ are independent sets in $L$, so $x$ is adjacent to the edges in $N_L(x)^+$ and $N_L(x)^-$
Any vertex $y$ in $L-V(C)$ other than $x$ must also be isolated, by \cref{l:eachCompTrivial}. 
Furthermore $y$ cannot be adjacent to two vertices in $N_L(x)^+$ or two vertices in $ N_L(x)^-$, as otherwise $L$ contains either the path $x_i^+ y x_j^+$ or the path $x_i^- y x_j^-$ for some distinct $i, j \in [d]$, contrary to \cref{l:propsOfLongestCycle}(ii).
Also if $uu^+ \in E(C)$ with $u,u^+ \notin N_L(x)$, then the edges $u$ and $u^+$ must be adjacent to $x$ but not each other. Therefore the edges $u$ and $u^+$ must be incident to different endpoints of $x$.

First we prove the following claim that simplifies the remainder of the proof.
\begin{claim}\label{cl:xi+xi-}
$x_i^+x_i^- \notin E(L)$ for all $i \in [d]$.
\end{claim}
\begin{proof}
Suppose for a contradiction and without loss of generality that $x_1^+x_1^- \in E(L)$. Then $x_1^+$ is incident to $v$ and $x_1^-$ is incident to $v'$. Since $N_L(x)^-$ is an independent set in $L$, $x_1^+ \neq x_{2}^-$ and therefore $x_1^{+2} \neq x_{2}$ and so $x_1^{+2}$ is incident to $v'$.
By \cref{l:nonXYedgesExtendCycle} with $(w,w')=(x_1^+,x_1^{+2})$ and 
$(z,z')=(x_1^-,x_j^+)$, $x_j^+$ is adjacent to $x_1^{+2}$ for all $j \neq 1$.
Also, by \cref{l:common_arg}$(j,1,1)$, $x_1$ is adjacent to both $x_j^+$ and $x_j^-$ for all $j \neq 1$. 
Therefore, given that $N_L(x)^+$ is an independent set in $L$ and $x_1$ and $x_1^+$ are not adjacent, $x_j^+$ is adjacent to $x$, $x_1$, $x_1^+$ and $x_1^{+2}$, which is possible only if $x_j^+$ is incident to $v$ and is adjacent to both $x_1$ and $x_1^{+2}$ on the same endpoint for all $j\neq 1$. Clearly only one such edge $x_j^+$ can exist. So $d=2$.
Furthermore
$x_{2}^+$ is not adjacent to $x_1^-$, which implies that $x_2^{+}\neq x_{1}^-$.

If $x_1^{-2} \neq x_{2}^+$, then $x_1^{-2}$ is incident to $v$ but is not $x_{2}^+$.
Then $x_1^{-2}$ is not adjacent to $x_1^{+2}$ and so 
\[
C' = xx_1 x_1^+ x_1^- x_{2}^+ \ora x_1^{-2} x_1^{+2} \ora x_2 x 
\]
is a cycle longer than $C$. So $x_1^{-2} = x_{2}^+$.
If $x_1^{+2} \neq x_2^-$, then $x_2^-$ must be adjacent to $x_1$ and $x_1^-$, which is possible only if $x_2^-$ is incident to $v'$. As $x_1^{+2} \neq x_2^-$, $x_2^-$ cannot be adjacent to $x_2^+$. Hence, given that $x_1^{+2}$ cannot be adjacent to $x_2$, 
\[
C' = x x_1 x_1^+ x_1^- x_2^+ x_2^- \ola x_1^{+2} x_2 x
\]
is a cycle longer than $C$. Consequently $ x_1^{+2} = x_2^-$.
Finally as $x_2$ cannot be adjacent to $x_2^-$ or $x_2^+$, $G$ must be either $H_1$, $H_2$ or $H_3$.
\end{proof}
By \cref{l:propsOfLongestCycle}(iii), $N_L(x)^+\cup \{x\}$ and $N_L(x)^-\cup \{x\}$ are each an independent set in $L$. Therefore as edges in $G$, each of $N_L(x)^+ \cup \{x\}$ and $N_L(x)^- \cup \{x\}$ must form either a star or a triangle, where the latter is possible only if $d=2$.
We separate the proof into cases based on which of these occurs, in each case yielding a contradiction.

\noindent \textbf{Case 1}: The vertices of $N_L(x)^+\cup\{x\}$ and $N_L(x)^-\cup \{ x\}$ correspond to stars on the same vertex. 

Exactly half the vertices  of $x_i \ora x_{i+1}^-$, as edges in $G$, are incident to $v$ for each $i \in [d]$, since the edges corresponding to the vertices in $x_i^+ \ora x_{i+1}^-$ alternate between being incident to $v$ and being incident to $v'$ and the first and final edges are incident to $v$. 
As every vertex in $C$ is in exactly one $x_i \ora x_{i+1}^-$ for some $i\in [d]$, it follows that exactly half of the edges corresponding to vertices in $C$ are incident to $v$.

We show that whenever there is an isolated vertex $y \neq x$ in $V(L)-V(C)$, the edge $y$ must be incident to $v$. Suppose otherwise that $y$ is incident to $v'$. 
The edge $y$ cannot be adjacent to $x_i^+$ for some $i$, as then $y$ is not adjacent to both $x_i^-$ and $x_i$, contradicting \cref{l:propsOfLongestCycle}. On the other hand
$y$ must be adjacent to all but at most 1 edge in $N_L(x)^+$ by \cref{l:propsOfLongestCycle}. Since $N_L(x)^+$ has at least two elements, we have a contradiction. Consequently, $y$ is incident to $v$.

Thus, there are 
\[
\frac{|V(C)|}{2}+|V(L)-V(C)|  = \frac{|V(L)|+(|V(L)|-|V(C)|)}{2}\geq \frac{|V(L)|+1}{2} = \frac{|E(G)|+1}{2}
\] 
edges in $G$ incident to $v$, since $|E(G)|=|V(L)|$ and $|V(L)|-|V(C)| \geq 1$.
However by \cref{l:disconColineGraph}(i), $L$ cannot be tough, since $G$ has a $K_{1,\Delta}$ subgraph and 
\[
|E(G)| < 2\left(\frac{|E(G)|+1}{2} \right)\leq 2\deg_G(v) \leq 2\Delta = \rho(K_{1,\Delta})
\]
edges, a contradiction.

$ $
\\
\noindent \textbf{Case 2}: $d\geq 3$ and the vertices of $N_L(x)^+ \cup \{x\}$ and $N_L(x)^- \cup \{x\}$ correspond to stars but on different endpoints of $x$.

Since $N_L(x)^+ \cup \{x\}$ and $N_L(x)^- \cup \{x\}$ correspond to stars on different endpoints of $x$, we must have that $N_L(x)^+ \cap N_L(x)^- =\emptyset$. 
By \cref{cl:xi+xi-}, the edge $x_j^+$ is adjacent to $x_j^-$ for all $j \in [d]$. Furthermore, because $x_j^+$ is adjacent to at most one edge in $N_L(x)^-$, we have that $x_{i+1}^- x_{j}^+ \in E(L)$ for all $i \in [d]$ and $j\neq i+1$.
Therefore \cref{l:common_arg}$(i,i+1,j)$ for $j \in [d]-\{i,i+1\}$ implies that the edges $x_i^+$ and $x_i^-$ are adjacent to all the elements in $N_L(x)-\{x_i\}$ 
for all $i\in [d]$.
Thus the edge $x_i$ is adjacent to every element in $N_L(x)^\pm-\{x_i^+,x_i^-\}$  for all $i\in [d]$. Since $x_j^+$ and $x_j^-$ are incident to a unique endpoint (that isn't an endpoint of $x$) for each $j$, $x_i$ can be adjacent to at most two pairs $x_j^+$ and $x_j^-$. Therefore, $d \leq 3$ and consequently $d=3$. Furthermore, the elements of $N_L(x)^\pm\cup \{x \}$, as edges, form a $K_5$.
As $G$ cannot be $K_5$, there is either another vertex in $V(L)-V(C)$ other than $x$ or a vertex in $C$ that isn't in  $N_L(x)^\pm $.
Any isolated vertex $y \neq x$ of $L-V(C)$ must, as an edge, be adjacent to $x$ and two elements from each of $N_L(x)^+$ and $N_L(x)^-$, by \cref{l:propsOfLongestCycle}(ii). However, this 
is impossible for an edge that is not $x$.
So there is a vertex in $C$ that is not in $N_L(x)^\pm$, say without loss of generality $x_1^{+2} \neq x_2^-$. Then the edge $x_1^{+2}$ is not adjacent to $x_2^+$ and so $x_1^{+2}x_2^+ \in E(L)$. Therefore as $x_1^+ x_3^{-} \in E(L)$ and $x_1^+ x_1^{+2} \in x_3\ora x_2$, \cref{l:nonXYedgesExtendCycle} applied to $(w,w')=(x_1^+,x_1^{+2})$ and $(z,z')=(x_2^+,x_3^-)$, yields a contradiction.
$ $\\

\noindent \textbf{Case 3}: $d=2$ and the edges
$N_L(x)^+ \cup \{x\}$ and $N_L(x)^- \cup \{x\}$ don't both form stars on the same vertex.

Since $N_L(x)^+ \cup \{x\}$ and $N_L(x)^- \cup \{x\}$ are independent sets in $L$ and $d=2$, the elements of $N_L(x)^+ \cup \{x\}$ and $N_L(x)^- \cup \{x\}$ as edges form either stars or triangles. By \cref{cl:xi+xi-} the edges $x_1^+$ and $x_1^-$ are adjacent as are the edges $x_2^+$ and $x_2^-$. Given both these constraints and noting that $N_L(x)^+$ and $N_L(x)^-$ both have cardinality 2 and their elements are incident to $v$ and $v'$ respectively, while it is possible that $x_1^+=x_2^-$ or $x_2^+=x_1^-$, it must be the case that the elements of $N_L(x)^\pm \cup \{x\}$ as edges in $G$
forms a graph $J$ that is isomorphic to $K_3$ when both $x_1^+=x_2^-$ and $x_2^+=x_1^-$, is isomorphic to $F_3$  when exactly one of $x_1^+=x_2^-$ and $x_2^+=x_1^-$ is true and is isomorphic to $K_4^-$ when $x_1^+\neq x_2^-$ and $x_2^+ \neq x_1^-$.
Without loss of generality, suppose that we don't have both $x_1^+ \neq x_2^-$ and $x_2^+=x_1^-$.

First suppose that $x_1^- $ is either $x_2^+$ or $x_2^{+2}$ and $x_2^- $ is either $x_1^+$ or $x_1^{+2}$. 
In particular, $V(C)-\{x_1,x_2\} = N_L(x)^\pm$ and so $(V(C)\cup\{x\})-\{x_1,x_2\}= E(J)$.
We show that the graph $J'$ formed from the edges $E(G)-\{x_1,x_2\}$  is isomorphic to either $K_{1,\lambda}$ with $\lambda\geq 3$, $F_k$ with $k \geq 2$ or $K_4^-$ as then $G$ has $J'$ as a subgraph and has  $|E(G)|= |E(J')|+2< \rho(J')$ edges, yet $L$ is tough contrary to \cref{l:disconColineGraph}(i), (iv) or (v).
The edges corresponding to the elements of $V(L)-V(C)$ must either form a triangle or a star and each of them must be adjacent to at least one element of $N_L(x)^+$ and one element of $N_L(x)^-$, by \cref{l:propsOfLongestCycle}(ii). 
So considering all possible graphs $J$ and all possible edges corresponding to vertices in $V(L)-V(C)$,  $J'$ isn't isomorphic to either $K_{1,\lambda}$ with $\lambda\geq 3$, $F_k$ with $k \geq 2$ or $K_4^-$ only if $V(L)-V(C) \neq \{x\}$ and 
$J$ is isomorphic to $K_4^-$. Furthermore, 
 $x_1^+$ and $x_1^-$ must be incident to $v$ and $x_2^+$ and $x_2^-$ must be incident to $v'$. 
Then any $y \neq x$ in $V(L)-V(C)$ is adjacent to both $x_i^-$ and $x_i^+$ in $L$ for some $i\in [2]$ and thus
$C' = x x_i x_{i}^+ y x_i^- x_{3-i}^+ x_{3-i} x$ is a cycle longer than $C$, since $x_{3-i}^-$ is the only vertex  from $C$ that isn't in $C'$, a contradiction.

Now suppose that, without loss of generality, $x_1^-$ is neither $ x_2^+$ or $x_2^{+2}$.
Let $z$ be the edge in $\{x_2^+,x_1^-\}$ incident to $v$, such an edge exists since $x_1^+ x_1^- \notin E(L)$ by \cref{cl:xi+xi-}. For any vertex $u \in x_2^{+2}\ora x_1^{-2}$, the edge $u$ cannot be adjacent to $x_1^+$ unless it is incident to $v$, because $u$ is neither $x_1^-$ or $x_2^+$. Consequently $u x_1^+ \in E(L)$ if $u$ is not incident to $v$. Thus, if $x_1^+ =x_2^-$, then 
\[
C' = 
\begin{cases}
 x x_1 x_1^+ x_2^{+2} \ora x_1^- x_2^+ x_2 x &\text{ if } z =x_2^+\\
 x x_1 x_1^- x_2^{+} \ora x_1^{-2} x_1^+ x_2 x &\text{ if } z =x_1^-
\end{cases}
\]
is a cycle longer than $C$, since the edges $x_2^{+2}$ and $x_1^{-2}$ are not incident to $v$ when $z=x_2^+$ and $z=x_1^{-}$, respectively because the edges in $x_2^{+2}\ora x_1^{-2}$ alternate endpoints of $x$. 
Therefore, $x_1^+ \neq x_2^-$.
Since $x_2^{+2}$ and $x_1^-$ are incident to the same endpoint of $x$, $x_2^{+3}$ cannot be $ x_1^-$. As the edges $x_2^{+2}$ and $x_{2}^{+3}$ are incident to different endpoints of $x$, with $x_2^+$ and $x_{2}^{+3}$ incident to the same endpoint of $x$ and $x_2^{+2},x_2^{+3} \notin N_L(x)^\pm$, either the edges corresponding to $N_L(x)^+\cup \{x\}$ and $N_L(x)^-\cup \{x\}$ form stars and $x_2^{+2} x_{1}^+  , x_2^{+3} x_2^- \in E(L)$ or the edges corresponding to $N_L(x)^+\cup \{x\}$ and $N_L(x)^-\cup \{x\}$ form triangles and $x_2^{+2} x_{2}^-  , x_2^{+3} x_{1}^+ \in E(L)$.
Either case contradicts \cref{l:nonXYedgesExtendCycle} applied to  $(w,w')=(x_2^{+2},x_2^{+3})$ and $\{z,z'\} = \{x_1^+,x_2^-\}$.
\end{proof}

\section{Toughness and Traceability of coline graphs}\label{s:TNT}
Now we consider traceability and toughness of coline graphs.
Combining Theorems~\ref{t:WuMe} and \ref{t:main} and \cref{l:Hs}, the coline graph $L$ of a graph $G$ is tough unless $G$ satisfies one of (i)-(iv) in \cref{t:WuMe} but isn't $H_1, H_2$ or $H_3$. Comparing this to the exceptions in Theorem 1 of~\cite{Li05}, gives the corollary below that provides a more explicit characterisation of non-tough coline graphs. Note that it is immediate from \cref{t:main} and the corollary below that the graphs satisfying (iii) and (iv) but not (i) or (ii) of \cref{t:WuMe} are exactly the 21 graphs in Figures~\ref{f:graphH} and~\ref{f:exceptions}. For completeness we give a proof of the corollary independent of Theorem~\ref{t:WuMe} and~\cite{Li05}.

\begin{cor}\label{c:exceptions}
Let $G$ be a graph with maximum degree $\Delta$. Then its coline graph $L$ is not tough if and only if $G$ either 
\begin{itemize}
\item[(i)]has less than $2\Delta$ edges;
\item[(ii)] has exactly $2\Delta$ edges and two vertices of degree $\Delta$ that are adjacent;
\item[(iii)]is one of the graphs in \cref{f:exceptions}.
\end{itemize}
\end{cor}
\begin{proof}
The reverse implication is immediate for (i) and (ii) from \cref{l:disconColineGraph}(i) and (ii), respectively, since any graph satisfying (ii) must be type-$I$. The fact that no graph in \cref{f:exceptions} is tough is immediate from each of the cutsets depicted  and \cref{l:disconColineGraph}(iii)-(vi).

For the forward implication we assume that $G$ has a non-tough coline graph $L$ which doesn't satisfy (i) or (ii).
Then $L$ has a cutset $S$ for which $c(L-S) >|S|$ and so  $G'=G-S$, the graph formed from $G$ by removing the edges in $S$, must have a disconnected coline graph, namely $L-S$.
So $G'$ must be an exceptional graph and as $L$ doesn't satisfy (i) or (ii), $G'$ satisfies one of (iii)-(vi) of \cref{l:disconColineGraph}.
So we consider each of these in turn. Let $G''$ be the subgraph of $G$ consisting of the edges in $E(G)-E(G')$. 

If $G'=C_4$, then $G$ must have at most $5 = \rho(C_4)-1$ edges by \cref{l:disconColineGraph}(iii) and at least 5 edges, since $G$ doesn't satisfy (ii). So $G$ must have exactly 5 edges and $G''$ must consist of a single edge disjoint from $G'$, as otherwise, $G$ would satisfy (i). So $G=C_4 \cup K_2$, which is one of the graphs in \cref{f:exceptions}.

If $G' = F_k$ for some $k \geq 2$, then by \cref{l:disconColineGraph}(iv), $G$ must have at most $k+3=\rho(F_k)-1$ edges and must have at least $2k $ edges, because $G$ doesn't satisfy (i). 
So $2k\leq k+3$ which implies that $k \leq 3$. 
If $k=2$ and therefore $G$ has at most 5 edges, then  since $G$ does not satisfy (i), $5 \geq 2\Delta$. Hence, $\Delta \leq 2$ and as $F_3$ contains vertices of degree $2$, $\Delta \geq 2$. Consequently, $\Delta=2$. Then $G$ must in fact have exactly 5 edges as it doesn't satisfy (ii) and has two adjacent vertices of degree 2.
Hence,  $G''$ must consist of two edges disjoint from $G'$ as otherwise $G$ would satisfy (i). Thus $G$ is either $K_3 \cup 2K_2$ or $K_3 \cup P_3$, which are graphs in \cref{f:exceptions}. 
If $k=3$, then $G$ has maximum degree at least 3 and so has 6 edges because $G$ does not satisfy (i). 
Therefore, at most 1 edge in $G''$ is incident to a vertex in $G'$ and if such an edge exists, then it must be incident to the unique one in $G'$ with degree 1 as otherwise $G$ would either have maximum degree at least 4 or have two adjacent vertices of degree 3, contrary to $G$ not satisfying (i) and (ii), respectively. 
Thus $G$ must be one of the 4 graphs in the second last row of \cref{f:exceptions}.

If $G' = K_4^-$, then $G$ must have at least 7 edges as otherwise it satisfies either (i) or (ii). By \cref{l:disconColineGraph}(v), $G$ has at most 7 edges and so $G$ has exactly 7 edges. No edge in $G''$ can be incident to one of the vertices of degree 3 in $G'$ as otherwise $G$ would satisfy (i). 
Similarly, only 1 edge in $G''$ can be adjacent to each of the 2 vertices in $G'$ of degree 2. Thus $G$ must be one of the 6 graphs in \cref{f:exceptions} whose exceptional graph is $K_4^-$ or $G$ must be $K_4\cup K_2$.

Finally if $G'=K_4$, then $G$ must have at least 7 edges as it doesn't satisfy (ii) and has at most 8 edges by \cref{l:disconColineGraph}(vi). If $G$ has 7 edges, then $G$ has maximum degree at least 4, and would therefore satisfy (i), unless $G = K_4\cup K_2$, which is one of the graphs in \cref{f:exceptions}. If $G$ has 8 edges, then at most one edge in $G''$ is incident to a vertex in $G'$ as otherwise $G$ would have maximum degree 5 or have two adjacent vertices of maximum degree 4, contrary to $G$ not satisfying (i) and (ii), respectively. Thus $G$ must be one of the first 4 graphs in \cref{f:exceptions}.

\end{proof}

\begin{figure}[htbp]
\centering 
\begin{subfigure}[b]{0.15\textwidth}
\begin{tikzpicture}[scale=\scale]
 \draw[line width = \lwid, color = blue]{

(\r,\r)-- (-\r,\r)
(-\r,\r)-- (-\r,-\r)
(-\r,-\r)-- (\r,-\r)
(\r,-\r)-- (\r,\r)
(\r,\r)-- (-\r,-\r)
(-\r,\r)--(\r,-\r)
};
 \draw[line width = \lwid, color = Green]{
(2*\r,-\r)--(2*\r,\r)
(3*\r,-\r)--(3*\r,\r)
};
\draw[line width =\cwid, color = black!50]{
(\r,\r)  node[circle, draw, fill=black!10,inner sep=\inse, minimum width=\inci] {}
(\r,-\r) node[circle, draw, fill=black!10,inner sep=\inse, minimum width=\inci] {}
(-\r,\r) node[circle, draw, fill=black!10,inner sep=\inse, minimum width=\inci] {}
(-\r,-\r)  node[circle, draw, fill=black!10,inner sep=\inse, minimum width=\inci] {}
(2*\r,-\r)  node[circle, draw, fill=black!10,inner sep=\inse, minimum width=\inci] {}
(2*\r,\r)  node[circle, draw, fill=black!10,inner sep=\inse, minimum width=\inci] {}
(3*\r,\r) node[circle, draw, fill=black!10,inner sep=\inse, minimum width=\inci] {}
(3*\r,-\r) node[circle, draw, fill=black!10,inner sep=\inse, minimum width=\inci] {}
};
\end{tikzpicture}
\end{subfigure}
\hspace*{5mm}
\begin{subfigure}[b]{0.15\textwidth}
\begin{tikzpicture}[scale=\scale]
 \draw[line width = \lwid, color = blue]{

(\r,\r)-- (-\r,\r)
(-\r,\r)-- (-\r,-\r)
(-\r,-\r)-- (\r,-\r)
(\r,-\r)-- (\r,\r)
(\r,\r)-- (-\r,-\r)
(-\r,\r)--(\r,-\r)
};
 \draw[line width = \lwid, color = Green]{
(2*\r,-\r)--(2.5*\r,\r)
(2.5*\r,\r)--(3*\r,-\r)
};
\draw[line width = \cwid, color = black!50]{
(\r,\r)  node[circle, draw, fill=black!10,inner sep=\inse, minimum width=\inci] {}
(\r,-\r) node[circle, draw, fill=black!10,inner sep=\inse, minimum width=\inci] {}
(-\r,\r) node[circle, draw, fill=black!10,inner sep=\inse, minimum width=\inci] {}
(-\r,-\r)  node[circle, draw, fill=black!10,inner sep=\inse, minimum width=\inci] {}
(2*\r,-\r)  node[circle, draw, fill=black!10,inner sep=\inse, minimum width=\inci] {}
(3*\r,-\r) node[circle, draw, fill=black!10,inner sep=\inse, minimum width=\inci] {}
(2.5*\r,\r)node[circle, draw, fill=black!10,inner sep=\inse, minimum width=\inci] {}
};
\end{tikzpicture}
\end{subfigure} 
\hspace*{5mm}
\begin{subfigure}[b]{0.15\textwidth}
\begin{tikzpicture}[scale =\scale]
 \draw[line width = \lwid, color = blue]{

(\r,\r)-- (-\r,\r)
(-\r,\r)-- (-\r,-\r)
(-\r,-\r)-- (\r,-\r)
(\r,-\r)-- (\r,\r)
(\r,\r)-- (-\r,-\r)
(-\r,\r)--(\r,-\r)
};
 \draw[line width = \lwid, color = Green]{
(\r,-\r)--(2.73*\r,0)
(3.73*\r,-\r)--(3.73*\r,\r)
};
\draw[line width = \cwid, color = black!50]{
(\r,\r)  node[circle, draw, fill=black!10,inner sep=\inse, minimum width=\inci] {}
(\r,-\r) node[circle, draw, fill=black!10,inner sep=\inse, minimum width=\inci] {}
(-\r,\r) node[circle, draw, fill=black!10,inner sep=\inse, minimum width=\inci] {}
(-\r,-\r)  node[circle, draw, fill=black!10,inner sep=\inse, minimum width=\inci] {}
(3.73*\r,-\r)  node[circle, draw, fill=black!10,inner sep=\inse, minimum width=\inci] {}
(3.73*\r,\r)  node[circle, draw, fill=black!10,inner sep=\inse, minimum width=\inci] {}
(2.73*\r,0) node[circle, draw, fill=black!10,inner sep=\inse, minimum width=\inci] {}
};
\end{tikzpicture}
\end{subfigure}
\hspace*{5mm}
\begin{subfigure}[b]{0.15\textwidth}
\begin{tikzpicture}[scale=\scale]
 \draw[line width = \lwid, color = blue]{

(\r,\r)-- (-\r,\r)
(-\r,\r)-- (-\r,-\r)
(-\r,-\r)-- (\r,-\r)
(\r,-\r)-- (\r,\r)
(\r,\r)-- (-\r,-\r)
(-\r,\r)--(\r,-\r)
};
 \draw[line width = \lwid, color = Green]{
(\r,-\r)--(3*\r,-\r)
(3*\r,-\r)--(3*\r,\r)
};
\draw[line width = \cwid, color = black!50]{
(\r,\r)  node[circle, draw, fill=black!10,inner sep=\inse, minimum width=\inci] {}
(\r,-\r) node[circle, draw, fill=black!10,inner sep=\inse, minimum width=\inci] {}
(-\r,\r) node[circle, draw, fill=black!10,inner sep=\inse, minimum width=\inci] {}
(-\r,-\r)  node[circle, draw, fill=black!10,inner sep=\inse, minimum width=\inci] {}
(3*\r,-\r)  node[circle, draw, fill=black!10,inner sep=\inse, minimum width=\inci] {}
(3*\r,\r)  node[circle, draw, fill=black!10,inner sep=\inse, minimum width=\inci] {}
};
\end{tikzpicture}
\end{subfigure}
\hspace*{5mm}
\begin{subfigure}[b]{0.15\textwidth}
\begin{tikzpicture}[scale=\scale]
 \draw[line width = \lwid, color = blue]{

(\r,\r)-- (-\r,\r)
(-\r,\r)-- (-\r,-\r)
(-\r,-\r)-- (\r,-\r)
(\r,-\r)-- (\r,\r)
(\r,\r)-- (-\r,-\r)
(-\r,\r)--(\r,-\r)
};
 \draw[line width = \lwid, color = Green]{
(2*\r,-\r)--(2*\r,\r)
};
\draw[line width = \cwid, color = black!50]{
(\r,\r)  node[circle, draw, fill=black!10,inner sep=\inse, minimum width=\inci] {}
(\r,-\r) node[circle, draw, fill=black!10,inner sep=\inse, minimum width=\inci] {}
(-\r,\r) node[circle, draw, fill=black!10,inner sep=\inse, minimum width=\inci] {}
(-\r,-\r) node[circle, draw, fill=black!10,inner sep=\inse, minimum width=\inci] {}
(2*\r,-\r) node[circle, draw, fill=black!10,inner sep=\inse, minimum width=\inci] {}
(2*\r,\r) node[circle, draw, fill=black!10,inner sep=\inse, minimum width=\inci] {}
};
\end{tikzpicture}
\end{subfigure}
\\
\vspace*{3mm}
\begin{subfigure}[b]{0.15\textwidth}
\begin{tikzpicture}[scale=\scale]
 \draw[line width = \lwid, color = blue]{

(\r,\r)-- (-\r,\r)
(-\r,\r)-- (-\r,-\r)
(-\r,-\r)-- (\r,-\r)
(\r,-\r)-- (\r,\r)
(\r,\r)-- (-\r,-\r)
};
 \draw[line width = \lwid, color = Green]{
(2*\r,-\r)--(2*\r,\r)
(3*\r,-\r)--(3*\r,\r)
};
\draw[line width = \cwid, color = black!50]{
(\r,\r)  node[circle, draw, fill=black!10,inner sep=\inse, minimum width=\inci] {}
(\r,-\r) node[circle, draw, fill=black!10,inner sep=\inse, minimum width=\inci] {}
(-\r,\r) node[circle, draw, fill=black!10,inner sep=\inse, minimum width=\inci] {}
(-\r,-\r)  node[circle, draw, fill=black!10,inner sep=\inse, minimum width=\inci] {}
(2*\r,-\r)  node[circle, draw, fill=black!10,inner sep=\inse, minimum width=\inci] {}
(2*\r,\r)  node[circle, draw, fill=black!10,inner sep=\inse, minimum width=\inci] {}
(3*\r,\r) node[circle, draw, fill=black!10,inner sep=\inse, minimum width=\inci] {}
(3*\r,-\r) node[circle, draw, fill=black!10,inner sep=\inse, minimum width=\inci] {}
};
\end{tikzpicture}
\end{subfigure}
\hspace*{5mm}
\begin{subfigure}[b]{0.15\textwidth}
\begin{tikzpicture}[scale=\scale]
 \draw[line width = \lwid, color = blue]{

(\r,\r)-- (-\r,\r)
(-\r,\r)-- (-\r,-\r)
(-\r,-\r)-- (\r,-\r)
(\r,-\r)-- (\r,\r)
(\r,\r)-- (-\r,-\r)
};
 \draw[line width = \lwid, color = Green]{
(2*\r,-\r)--(2.5*\r,\r)
(2.5*\r,\r)--(3*\r,-\r)
};
\draw[line width = \cwid, color = black!50]{
(\r,\r)  node[circle, draw, fill=black!10,inner sep=\inse, minimum width=\inci] {}
(\r,-\r) node[circle, draw, fill=black!10,inner sep=\inse, minimum width=\inci] {}
(-\r,\r) node[circle, draw, fill=black!10,inner sep=\inse, minimum width=\inci] {}
(-\r,-\r)  node[circle, draw, fill=black!10,inner sep=\inse, minimum width=\inci] {}
(2*\r,-\r)  node[circle, draw, fill=black!10,inner sep=\inse, minimum width=\inci] {}
(3*\r,-\r) node[circle, draw, fill=black!10,inner sep=\inse, minimum width=\inci] {}
(2.5*\r,\r)node[circle, draw, fill=black!10,inner sep=\inse, minimum width=\inci] {}
};
\end{tikzpicture}
\end{subfigure}
\hspace*{5mm}
\begin{subfigure}[b]{0.15\textwidth}
\begin{tikzpicture}[scale =\scale]
 \draw[line width = \lwid, color = blue]{

(\r,\r)-- (-\r,\r)
(-\r,\r)-- (-\r,-\r)
(-\r,-\r)-- (\r,-\r)
(\r,-\r)-- (\r,\r)
(\r,\r)-- (-\r,-\r)
};
 \draw[line width = \lwid, color = Green]{
(\r,-\r)--(2.73*\r,0)
(3.73*\r,-\r)--(3.73*\r,\r)
};
\draw[line width = \cwid, color = black!50]{
(\r,\r)  node[circle, draw, fill=black!10,inner sep=\inse, minimum width=\inci] {}
(\r,-\r) node[circle, draw, fill=black!10,inner sep=\inse, minimum width=\inci] {}
(-\r,\r) node[circle, draw, fill=black!10,inner sep=\inse, minimum width=\inci] {}
(-\r,-\r)  node[circle, draw, fill=black!10,inner sep=\inse, minimum width=\inci] {}
(3.73*\r,-\r)  node[circle, draw, fill=black!10,inner sep=\inse, minimum width=\inci] {}
(3.73*\r,\r)  node[circle, draw, fill=black!10,inner sep=\inse, minimum width=\inci] {}
(2.73*\r,0) node[circle, draw, fill=black!10,inner sep=\inse, minimum width=\inci] {}
};
\end{tikzpicture}
\end{subfigure}
\hspace*{5mm}
\begin{subfigure}[b]{0.15\textwidth}
\begin{tikzpicture}[scale=\scale]
 \draw[line width = \lwid, color = blue]{

(\r,\r)-- (-\r,\r)
(-\r,\r)-- (-\r,-\r)
(-\r,-\r)-- (\r,-\r)
(\r,-\r)-- (\r,\r)
(\r,\r)-- (-\r,-\r)
};
 \draw[line width = \lwid, color = Green]{
(\r,-\r)--(3*\r,-\r)
(3*\r,-\r)--(3*\r,\r)
};
\draw[line width = \cwid, color = black!50]{
(\r,\r)  node[circle, draw, fill=black!10,inner sep=\inse, minimum width=\inci] {}
(\r,-\r) node[circle, draw, fill=black!10,inner sep=\inse, minimum width=\inci] {}
(-\r,\r) node[circle, draw, fill=black!10,inner sep=\inse, minimum width=\inci] {}
(-\r,-\r)  node[circle, draw, fill=black!10,inner sep=\inse, minimum width=\inci] {}
(3*\r,-\r)  node[circle, draw, fill=black!10,inner sep=\inse, minimum width=\inci] {}
(3*\r,\r)  node[circle, draw, fill=black!10,inner sep=\inse, minimum width=\inci] {}
};
\end{tikzpicture}   
\end{subfigure}
\hspace*{5mm}
\begin{subfigure}[b]{0.15\textwidth}
\begin{tikzpicture}[scale=\scale]
 \draw[line width = \lwid, color = blue]{

(\r,\r)-- (-\r,\r)
(-\r,\r)-- (-\r,-\r)
(-\r,-\r)-- (\r,-\r)
(\r,-\r)-- (\r,\r)
};
 \draw[line width = \lwid, color = Green]{
(2*\r,-\r)--(2*\r,\r)
};
\draw[line width = \cwid, color = black!50]{
(\r,\r)  node[circle, draw, fill=black!10,inner sep=\inse, minimum width=\inci] {}
(\r,-\r) node[circle, draw, fill=black!10,inner sep=\inse, minimum width=\inci] {}
(-\r,\r) node[circle, draw, fill=black!10,inner sep=\inse, minimum width=\inci] {}
(-\r,-\r)  node[circle, draw, fill=black!10,inner sep=\inse, minimum width=\inci] {}
(2*\r,-\r)  node[circle, draw, fill=black!10,inner sep=\inse, minimum width=\inci] {}
(2*\r,\r)  node[circle, draw, fill=black!10,inner sep=\inse, minimum width=\inci] {}

};
\end{tikzpicture}
\end{subfigure}
\\
\vspace*{3mm}
\hspace*{-2.5mm}
\begin{subfigure}[b]{0.15\textwidth}
\begin{tikzpicture}[scale=\scale]
 \draw[line width = \lwid, color = blue]{
(\r,\r)-- (-\r,\r)
(-\r,\r)-- (-\r,-\r)
(-\r,-\r)-- (\r,-\r)
(\r,\r)-- (-\r,-\r)
};
 \draw[line width = \lwid, color = Green]{
(2*\r,-\r)--(2*\r,\r)
(3*\r,-\r)--(3*\r,\r)
};
\draw[line width = \cwid, color = black!50]{
(\r,\r)  node[circle, draw, fill=black!10,inner sep=\inse, minimum width=\inci] {}
(\r,-\r) node[circle, draw, fill=black!10,inner sep=\inse, minimum width=\inci] {}
(-\r,\r) node[circle, draw, fill=black!10,inner sep=\inse, minimum width=\inci] {}
(-\r,-\r)  node[circle, draw, fill=black!10,inner sep=\inse, minimum width=\inci] {}
(2*\r,-\r)  node[circle, draw, fill=black!10,inner sep=\inse, minimum width=\inci] {}
(2*\r,\r)  node[circle, draw, fill=black!10,inner sep=\inse, minimum width=\inci] {}
(3*\r,\r) node[circle, draw, fill=black!10,inner sep=\inse, minimum width=\inci] {}
(3*\r,-\r) node[circle, draw, fill=black!10,inner sep=\inse, minimum width=\inci] {}
};
\end{tikzpicture}
\end{subfigure}
\hspace*{5mm}
\begin{subfigure}[b]{0.15\textwidth}
\begin{tikzpicture}[scale=\scale]
 \draw[line width = \lwid, color = blue]{

(\r,\r)-- (-\r,\r)
(-\r,\r)-- (-\r,-\r)
(-\r,-\r)-- (\r,-\r)
(\r,\r)-- (-\r,-\r)
};
 \draw[line width = \lwid, color = Green]{
(2*\r,-\r)--(2.5*\r,\r)
(2.5*\r,\r)--(3*\r,-\r)
};
\draw[line width = \cwid, color = black!50]{
(\r,\r)  node[circle, draw, fill=black!10,inner sep=\inse, minimum width=\inci] {}
(\r,-\r) node[circle, draw, fill=black!10,inner sep=\inse, minimum width=\inci] {}
(-\r,\r) node[circle, draw, fill=black!10,inner sep=\inse, minimum width=\inci] {}
(-\r,-\r)  node[circle, draw, fill=black!10,inner sep=\inse, minimum width=\inci] {}
(2*\r,-\r)  node[circle, draw, fill=black!10,inner sep=\inse, minimum width=\inci] {}
(3*\r,-\r) node[circle, draw, fill=black!10,inner sep=\inse, minimum width=\inci] {}
(2.5*\r,\r)node[circle, draw, fill=black!10,inner sep=\inse, minimum width=\inci] {}
};
\end{tikzpicture}
\end{subfigure}
\hspace*{5mm}
\begin{subfigure}[b]{0.15\textwidth}
\begin{tikzpicture}[scale =\scale]
 \draw[line width = \lwid, color = blue]{

(\r,\r)-- (-\r,\r)
(-\r,\r)-- (-\r,-\r)
(-\r,-\r)-- (\r,-\r)
(\r,\r)-- (-\r,-\r)
};
 \draw[line width = \lwid, color = Green]{
(\r,-\r)--(2.73*\r,0)
(3.73*\r,-\r)--(3.73*\r,\r)
};
\draw[line width = \cwid, color = black!50]{
(\r,\r)  node[circle, draw, fill=black!10,inner sep=\inse, minimum width=\inci] {}
(\r,-\r) node[circle, draw, fill=black!10,inner sep=\inse, minimum width=\inci] {}
(-\r,\r) node[circle, draw, fill=black!10,inner sep=\inse, minimum width=\inci] {}
(-\r,-\r)  node[circle, draw, fill=black!10,inner sep=\inse, minimum width=\inci] {}
(3.73*\r,-\r)  node[circle, draw, fill=black!10,inner sep=\inse, minimum width=\inci] {}
(3.73*\r,\r)  node[circle, draw, fill=black!10,inner sep=\inse, minimum width=\inci] {}
(2.73*\r,0) node[circle, draw, fill=black!10,inner sep=\inse, minimum width=\inci] {}
};
\end{tikzpicture}
\end{subfigure}
\hspace*{5mm}
\begin{subfigure}[b]{0.15\textwidth}
\begin{tikzpicture}[scale=\scale]
\draw[line width = \cwid, color = black!0]{
(-1.25*\r,-\r)  node[circle, draw, fill=black!0,inner sep=\inse, minimum width=\inci] {}
};
 \draw[line width = \lwid, color = blue]{

(\r,\r)-- (-\r,\r)
(-\r,\r)-- (-\r,-\r)
(-\r,-\r)-- (\r,-\r)
(\r,\r)-- (-\r,-\r)
};
 \draw[line width = \lwid, color = Green]{
(\r,-\r)--(3*\r,-\r)
(3*\r,-\r)--(3*\r,\r)
};
\draw[line width = \cwid, color = black!50]{
(\r,\r)  node[circle, draw, fill=black!10,inner sep=\inse, minimum width=\inci] {}
(\r,-\r) node[circle, draw, fill=black!10,inner sep=\inse, minimum width=\inci] {}
(-\r,\r) node[circle, draw, fill=black!10,inner sep=\inse, minimum width=\inci] {}
(-\r,-\r)  node[circle, draw, fill=black!10,inner sep=\inse, minimum width=\inci] {}
(3*\r,-\r)  node[circle, draw, fill=black!10,inner sep=\inse, minimum width=\inci] {}
(3*\r,\r)  node[circle, draw, fill=black!10,inner sep=\inse, minimum width=\inci] {}
};

\end{tikzpicture}   
\end{subfigure}
\hspace*{5mm}
\begin{subfigure}[b]{0.15\textwidth}
\begin{tikzpicture}[scale=\scale]
\draw[line width = \cwid, color = black!00]{
(\r,\r)  node[circle, draw, fill=black!0,inner sep=\inse, minimum width=\inci] {}
(\r,-\r) node[circle, draw, fill=black!0,inner sep=\inse, minimum width=\inci] {}
(-\r,\r) node[circle, draw, fill=black!0,inner sep=\inse, minimum width=\inci] {}
(-\r,-\r)  node[circle, draw, fill=black!0,inner sep=\inse, minimum width=\inci] {}
(2*\r,-\r)  node[circle, draw, fill=black!0,inner sep=\inse, minimum width=\inci] {}
(2*\r,\r)  node[circle, draw, fill=black!0,inner sep=\inse, minimum width=\inci] {}
};
\end{tikzpicture}
\end{subfigure}
\\
\vspace{3mm}
\hspace*{-15.75mm}
\begin{subfigure}[b]{0.15\textwidth}
\begin{tikzpicture}[scale=\scale]
\draw[line width = \cwid, color = black!0]{
(-2.73*\r,0)  node[circle, draw, fill=black!00,inner sep=\inse, minimum width=\inci] {}
};
\draw[line width = \lwid, color = blue]{

(-\r,-\r)-- (\r,-\r)
(-\r,-\r)-- (0,\r)
(\r,-\r)-- (0,1*\r)
};
 \draw[line width = \lwid, color = Green]{
(2*\r,-\r)--(2*\r,1*\r)
(3*\r,-\r)--(3*\r,1*\r)
};
\draw[line width = \cwid, color = black!50]{
(-\r,-\r)  node[circle, draw, fill=black!10,inner sep=\inse, minimum width=\inci] {}
(\r,-\r) node[circle, draw, fill=black!10,inner sep=\inse, minimum width=\inci] {}
(0,1*\r) node[circle, draw, fill=black!10,inner sep=\inse, minimum width=\inci] {}
(-\r,-\r)  node[circle, draw, fill=black!10,inner sep=\inse, minimum width=\inci] {}
(2*\r,-\r)  node[circle, draw, fill=black!10,inner sep=\inse, minimum width=\inci] {}
(2*\r,1*\r)  node[circle, draw, fill=black!10,inner sep=\inse, minimum width=\inci] {}
(3*\r,-\r)  node[circle, draw, fill=black!10,inner sep=\inse, minimum width=\inci] {}
(3*\r,1*\r)  node[circle, draw, fill=black!10,inner sep=\inse, minimum width=\inci] {}
};
\draw[line width = \cwid, color = black!00]{
(-2.0*\r,-2.0*\r)  node[circle, draw, fill=black!00,inner sep=\inse, minimum width=\inci] {}
};
\end{tikzpicture}
\end{subfigure}
\hspace*{5mm}
\begin{subfigure}[b]{0.15\textwidth}
\begin{tikzpicture}[scale=\scale]
\draw[line width = \cwid, color = black!0]{
(-2.73*\r,0)  node[circle, draw, fill=black!00,inner sep=\inse, minimum width=\inci] {}
};
 \draw[line width = \lwid, color = blue]{

(-\r,-\r)-- (\r,-\r)
(-\r,-\r)-- (0,1*\s)
(\r,-\r)-- (0,1*\r)
};
 \draw[line width = \lwid, color = Green]{
(2*\r,-\r)--(2.5*\r,1*\r)
(3*\r,-\r)--(2.5*\r,1*\r)
};
\draw[line width = \cwid, color = black!50]{
(-\r,-\r)  node[circle, draw, fill=black!10,inner sep=\inse, minimum width=\inci] {}
(\r,-\r) node[circle, draw, fill=black!10,inner sep=\inse, minimum width=\inci] {}
(0,1*\r) node[circle, draw, fill=black!10,inner sep=\inse, minimum width=\inci] {}
(-\r,-\r)  node[circle, draw, fill=black!10,inner sep=\inse, minimum width=\inci] {}
(2*\r,-\r)  node[circle, draw, fill=black!10,inner sep=\inse, minimum width=\inci] {}
(2.5*\r,1*\r)  node[circle, draw, fill=black!10,inner sep=\inse, minimum width=\inci] {}
(3*\r,-\r)  node[circle, draw, fill=black!10,inner sep=\inse, minimum width=\inci] {}
};
\draw[line width = \cwid, color = black!00]{
(-2.0*\r,-2.0*\r)  node[circle, draw, fill=black!00,inner sep=\inse, minimum width=\inci] {}
};
\end{tikzpicture}
\end{subfigure}
\hspace*{4.75mm}
\begin{subfigure}[b]{0.15\textwidth}
\begin{tikzpicture}[scale=\scale]
\draw[line width = \cwid, color = black!0]{
(-2.75*\r,0)  node[circle, draw, fill=black!00,inner sep=\inse, minimum width=\inci] {}
};
 \draw[line width = \lwid, color = blue]{
(\r,\r)-- (-\r,\r)
(-\r,\r)-- (-\r,-\r)
(-\r,-\r)-- (\r,-\r)
(\r,-\r)-- (\r,\r)
(\r,\r)-- (-\r,-\r)
};
 \draw[line width = \lwid, color = Green]{

(-\r,\r)--(-2.73*\r,0)
(\r,-\r)--(2.73*\r,0)
};
\draw[line width = \cwid, color = black!50]{
(\r,\r)  node[circle, draw, fill=black!10,inner sep=\inse, minimum width=\inci] {}
(\r,-\r) node[circle, draw, fill=black!10,inner sep=\inse, minimum width=\inci] {}
(-\r,\r) node[circle, draw, fill=black!10,inner sep=\inse, minimum width=\inci] {}
(-\r,-\r)  node[circle, draw, fill=black!10,inner sep=\inse, minimum width=\inci] {}
(-2.73*\r,0)  node[circle, draw, fill=black!10,inner sep=\inse, minimum width=\inci] {}
(2.73*\r,0)  node[circle, draw, fill=black!10,inner sep=\inse, minimum width=\inci] {}
};
\draw[line width = \cwid, color = black!00]{
(-2.0*\r,-2.0*\r)  node[circle, draw, fill=black!00,inner sep=\inse, minimum width=\inci] {}
};
\end{tikzpicture}
\end{subfigure}
\hspace*{4.75mm}
\begin{subfigure}[b]{0.15\textwidth}
\begin{tikzpicture}[scale=\scale]
\draw[line width = \cwid, color = black!0]{
(-2.73*\r,0)  node[circle, draw, fill=black!00,inner sep=\inse, minimum width=\inci] {}
};
 \draw[line width = \lwid, color = blue]{

(\r,\r)-- (-\r,\r)
(-\r,\r)-- (-\r,-\r)
(-\r,-\r)-- (\r,-\r)
(\r,-\r)-- (\r,\r)
(\r,\r)-- (-\r,-\r)
};
 \draw[line width = \lwid, color = Green]{
[bend left] (\r,-\r) to (-2.0*\r,-2.0*\r)
[bend right] (-\r,\r) to (-2.0*\r,-2.0*\r)
};
\draw[line width = \cwid, color = black!50]{
(\r,\r)  node[circle, draw, fill=black!10,inner sep=\inse, minimum width=\inci] {}
(\r,-\r) node[circle, draw, fill=black!10,inner sep=\inse, minimum width=\inci] {}
(-\r,\r) node[circle, draw, fill=black!10,inner sep=\inse, minimum width=\inci] {}
(-\r,-\r)  node[circle, draw, fill=black!10,inner sep=\inse, minimum width=\inci] {}
(-2.0*\r,-2.0*\r)  node[circle, draw, fill=black!10,inner sep=\inse, minimum width=\inci] {}
};
\end{tikzpicture}
    \end{subfigure}
\hspace*{5mm}
\begin{subfigure}[b]{0.15\textwidth}
\begin{tikzpicture}[scale=\scale]
\draw[line width = \cwid, color = black!0]{
(-2.73*\r,0)  node[circle, draw, fill=black!00,inner sep=\inse, minimum width=\inci] {}
};
\draw[line width = \cwid, color = black!00]{
(\r,\r)  node[circle, draw, fill=black!0,inner sep=\inse, minimum width=\inci] {}
(\r,-\r) node[circle, draw, fill=black!0,inner sep=\inse, minimum width=\inci] {}
(-\r,\r) node[circle, draw, fill=black!0,inner sep=\inse, minimum width=\inci] {}
(-\r,-\r)  node[circle, draw, fill=black!0,inner sep=\inse, minimum width=\inci] {}
(2*\r,-\r)  node[circle, draw, fill=black!0,inner sep=\inse, minimum width=\inci] {}
(2*\r,\r)  node[circle, draw, fill=black!0,inner sep=\inse, minimum width=\inci] {}
};
\end{tikzpicture}
\end{subfigure}
\caption{The 18 exceptions of \cref{c:exceptions} with the largest exceptional subgraph in blue and a cutset $S$ of their coline graph $L$ such that $c(L-S)>|S|$ in green.}
\label{f:exceptions}
\end{figure}
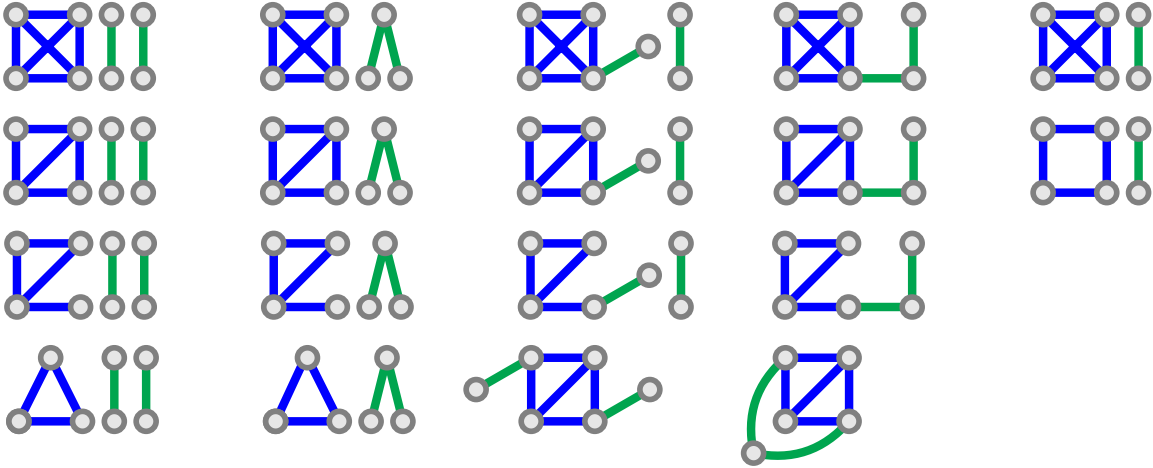

Now we consider traceability of coline graphs. 
We begin by proving \cref{c:main}, for which we need the following result by Whitney~\cite{Wh32}. 
\begin{thm}[\cite{Wh32}]\label{t:Whit}
Let $G_1$ and $G_2$ be distinct connected graphs whose line graphs are isomorphic. Then 
$\{G_1,G_2\} = \{K_3,K_{1,3}\}$.
\end{thm}
\begin{proof}[Proof of \cref{c:main}]
Let $G$ be a graph with a pseudo-tough coline graph $L$. Then $L^*$ is tough and so, by \cref{t:main}, $L^*$ is Hamiltonian unless it is $\co(K_5)$ or $\co(H_i)$ for some $i\in [3]$.
So as $L$ is traceable if and only if $L^*$ is Hamiltonian, $L$ is traceable unless it is either $\co(K_5)$ or $\co(H_i)$ for some $i\in [3]$ with a dominating vertex removed. As $\co(K_5)$, $\co(H_1)$ and $\co(H_2)$ do not contain a dominating vertex, $L$ must be $\co(H_3)$ with its unique dominating vertex removed, which is $\co(K_3 \circ K_1)=K_3 \circ K_1$. 
Since $L=\co(K_3 \circ K_1)$, we have that $\overline{L} = L(K_3 \circ K_1)$. It is easy to check that $L(K_3 \circ K_1)$ is connected. Any graph $H$ with $\co(H)=L$ satisfies $L(H)=\overline{L}$. Given that $H$ has no trivial components, the only way $L(H)=\overline{L}$ can be connected is if $H$ is connected. Thus since $\overline{L}$ is not isomorphic to $K_3$, 
Whitney's Theorem implies that $H \cong K_3 \circ K_1$.
\end{proof}

We also provide a characterisation of non-traceable graphs similar to \cref{c:exceptions}.
\begin{cor}\label{c:Trace}
Let $G$ be a graph with $m$ edges and maximum degree $\Delta$. Then $\co(G)$ is not traceable if and only if it satisfies one of the following
\begin{itemize}
\item[(i)] $m < 2\Delta-1$;
\item[(ii)] $m=2\Delta-1$ and two vertices of degree $\Delta$ are adjacent in $G$;
\item[(iii)] $G$ is one of the graphs in \cref{f:exceptionsTrace};
\item[(iv)] $G$ is $K_3\circ K_1$.
\end{itemize}
\end{cor}
\begin{proof}
Let $L=\co(G)$ and first consider the reverse implication.
If $G$ satisfies (i), (ii) or (iii), then $G \cup K_2$ satisfies (i), (ii) or (iii) of \cref{c:exceptions}, respectively since $|E(G \cup K_2)|=|E(G)|+1$. Then $\co(G\cup K_2)=L^*$ is not tough and so is not Hamiltonian, from which it follows that $L$ is not traceable. If $G = K_3 \circ K_1$, then by \cref{c:main}, $G$ is not traceable.

Now we consider the forward implication and suppose that $\co(G)$ is not traceable.
Then by \cref{c:main}, $L$ is not pseudo-tough or $G$ is $K_3\circ K_1$; the latter implies (iv) is true so suppose the former is true. As $L$ is pseudo-tough if and only if $L^*$ is tough, $L^*$ is not tough and \cref{c:exceptions} implies that the coline root graph $G \cup K_2$ of $L^*$ satisfies one of (i)-(iii) of \cref{c:exceptions}. 
If $G'$ satisfies (i) or (ii) of  \cref{c:exceptions} then clearly $G$ must satisfy (i) or (ii) of the corollary, respectively.
If $G'$ is one of the graphs in \cref{f:exceptions} it must have an isolated edge and so $G$ must be one of the graphs in  \cref{f:exceptionsTrace}.
\end{proof}

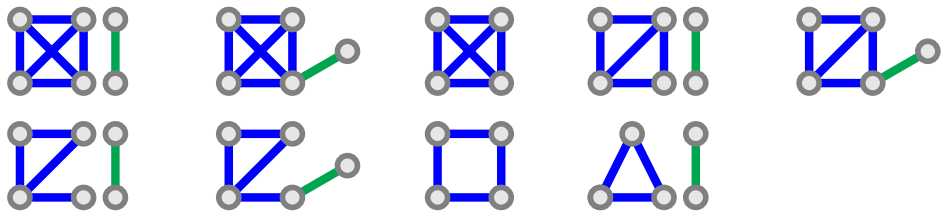
\begin{figure}[htbp]
\centering 
\begin{subfigure}[b]{0.15\textwidth}
\begin{tikzpicture}[scale=\scale]
 \draw[line width = \lwid, color = blue]{

(\r,\r)-- (-\r,\r)
(-\r,\r)-- (-\r,-\r)
(-\r,-\r)-- (\r,-\r)
(\r,-\r)-- (\r,\r)
(\r,\r)-- (-\r,-\r)
(-\r,\r)--(\r,-\r)
};
 \draw[line width = \lwid, color = Green]{
(2*\r,-\r)--(2*\r,\r)
};
\draw[line width =\cwid, color = black!50]{
(\r,\r)  node[circle, draw, fill=black!10,inner sep=\inse, minimum width=\inci] {}
(\r,-\r) node[circle, draw, fill=black!10,inner sep=\inse, minimum width=\inci] {}
(-\r,\r) node[circle, draw, fill=black!10,inner sep=\inse, minimum width=\inci] {}
(-\r,-\r)  node[circle, draw, fill=black!10,inner sep=\inse, minimum width=\inci] {}
(2*\r,-\r)  node[circle, draw, fill=black!10,inner sep=\inse, minimum width=\inci] {}
(2*\r,\r)  node[circle, draw, fill=black!10,inner sep=\inse, minimum width=\inci] {}
};
\draw[line width = \cwid, color = black!00]{
(2.73*\r,0) node[circle, draw, fill=black!0,inner sep=\inse, minimum width=\inci] {}
};
\end{tikzpicture}
\end{subfigure}
\begin{subfigure}[b]{0.15\textwidth}
\begin{tikzpicture}[scale =\scale]
 \draw[line width = \lwid, color = blue]{

(\r,\r)-- (-\r,\r)
(-\r,\r)-- (-\r,-\r)
(-\r,-\r)-- (\r,-\r)
(\r,-\r)-- (\r,\r)
(\r,\r)-- (-\r,-\r)
(-\r,\r)--(\r,-\r)
};
 \draw[line width = \lwid, color = Green]{
(\r,-\r)--(2.73*\r,0)
};
\draw[line width = \cwid, color = black!50]{
(\r,\r)  node[circle, draw, fill=black!10,inner sep=\inse, minimum width=\inci] {}
(\r,-\r) node[circle, draw, fill=black!10,inner sep=\inse, minimum width=\inci] {}
(-\r,\r) node[circle, draw, fill=black!10,inner sep=\inse, minimum width=\inci] {}
(-\r,-\r)  node[circle, draw, fill=black!10,inner sep=\inse, minimum width=\inci] {}
(2.73*\r,0) node[circle, draw, fill=black!10,inner sep=\inse, minimum width=\inci] {}
};
\end{tikzpicture}
\end{subfigure}
\begin{subfigure}[b]{0.15\textwidth}
\begin{tikzpicture}[scale=\scale]
 \draw[line width = \lwid, color = blue]{

(\r,\r)-- (-\r,\r)
(-\r,\r)-- (-\r,-\r)
(-\r,-\r)-- (\r,-\r)
(\r,-\r)-- (\r,\r)
(\r,\r)-- (-\r,-\r)
(-\r,\r)--(\r,-\r)
};
 \draw[line width = \lwid, color = Green]{
};
\draw[line width = \cwid, color = black!50]{
(\r,\r)  node[circle, draw, fill=black!10,inner sep=\inse, minimum width=\inci] {}
(\r,-\r) node[circle, draw, fill=black!10,inner sep=\inse, minimum width=\inci] {}
(-\r,\r) node[circle, draw, fill=black!10,inner sep=\inse, minimum width=\inci] {}
(-\r,-\r) node[circle, draw, fill=black!10,inner sep=\inse, minimum width=\inci] {}
};
\end{tikzpicture}
\end{subfigure}
\hspace*{-7.5mm} 
\begin{subfigure}[b]{0.15\textwidth}
\begin{tikzpicture}[scale=\scale]
 \draw[line width = \lwid, color = blue]{

(\r,\r)-- (-\r,\r)
(-\r,\r)-- (-\r,-\r)
(-\r,-\r)-- (\r,-\r)
(\r,-\r)-- (\r,\r)
(\r,\r)-- (-\r,-\r)
};
 \draw[line width = \lwid, color = Green]{
(2*\r,-\r)--(2*\r,\r)
};
\draw[line width = \cwid, color = black!50]{
(\r,\r)  node[circle, draw, fill=black!10,inner sep=\inse, minimum width=\inci] {}
(\r,-\r) node[circle, draw, fill=black!10,inner sep=\inse, minimum width=\inci] {}
(-\r,\r) node[circle, draw, fill=black!10,inner sep=\inse, minimum width=\inci] {}
(-\r,-\r)  node[circle, draw, fill=black!10,inner sep=\inse, minimum width=\inci] {}
(2*\r,-\r)  node[circle, draw, fill=black!10,inner sep=\inse, minimum width=\inci] {}
(2*\r,\r)  node[circle, draw, fill=black!10,inner sep=\inse, minimum width=\inci] {}
};
\end{tikzpicture}
\end{subfigure}
\begin{subfigure}[b]{0.15\textwidth}
\begin{tikzpicture}[scale =\scale]
 \draw[line width = \lwid, color = blue]{

(\r,\r)-- (-\r,\r)
(-\r,\r)-- (-\r,-\r)
(-\r,-\r)-- (\r,-\r)
(\r,-\r)-- (\r,\r)
(\r,\r)-- (-\r,-\r)
};
 \draw[line width = \lwid, color = Green]{
(\r,-\r)--(2.73*\r,0)
};
\draw[line width = \cwid, color = black!50]{
(\r,\r)  node[circle, draw, fill=black!10,inner sep=\inse, minimum width=\inci] {}
(\r,-\r) node[circle, draw, fill=black!10,inner sep=\inse, minimum width=\inci] {}
(-\r,\r) node[circle, draw, fill=black!10,inner sep=\inse, minimum width=\inci] {}
(-\r,-\r)  node[circle, draw, fill=black!10,inner sep=\inse, minimum width=\inci] {}
(2.73*\r,0) node[circle, draw, fill=black!10,inner sep=\inse, minimum width=\inci] {}
};
\end{tikzpicture}
\end{subfigure}
\\
\vspace*{3mm}
\begin{subfigure}[b]{0.15\textwidth}
\begin{tikzpicture}[scale=\scale]
 \draw[line width = \lwid, color = blue]{
(\r,\r)-- (-\r,\r)
(-\r,\r)-- (-\r,-\r)
(-\r,-\r)-- (\r,-\r)
(\r,\r)-- (-\r,-\r)
};
 \draw[line width = \lwid, color = Green]{
(2*\r,-\r)--(2*\r,\r)
};
\draw[line width = \cwid, color = black!50]{
(\r,\r)  node[circle, draw, fill=black!10,inner sep=\inse, minimum width=\inci] {}
(\r,-\r) node[circle, draw, fill=black!10,inner sep=\inse, minimum width=\inci] {}
(-\r,\r) node[circle, draw, fill=black!10,inner sep=\inse, minimum width=\inci] {}
(-\r,-\r)  node[circle, draw, fill=black!10,inner sep=\inse, minimum width=\inci] {}
(2*\r,-\r)  node[circle, draw, fill=black!10,inner sep=\inse, minimum width=\inci] {}
(2*\r,\r)  node[circle, draw, fill=black!10,inner sep=\inse, minimum width=\inci] {}

};
\draw[line width = \cwid, color = black!00]{
(2.73*\r,0) node[circle, draw, fill=black!0,inner sep=\inse, minimum width=\inci] {}
};
\end{tikzpicture}
\end{subfigure}
\begin{subfigure}[b]{0.15\textwidth}
\begin{tikzpicture}[scale =\scale]
 \draw[line width = \lwid, color = blue]{

(\r,\r)-- (-\r,\r)
(-\r,\r)-- (-\r,-\r)
(-\r,-\r)-- (\r,-\r)
(\r,\r)-- (-\r,-\r)
};
 \draw[line width = \lwid, color = Green]{
(\r,-\r)--(2.73*\r,0)
};
\draw[line width = \cwid, color = black!50]{
(\r,\r)  node[circle, draw, fill=black!10,inner sep=\inse, minimum width=\inci] {}
(\r,-\r) node[circle, draw, fill=black!10,inner sep=\inse, minimum width=\inci] {}
(-\r,\r) node[circle, draw, fill=black!10,inner sep=\inse, minimum width=\inci] {}
(-\r,-\r)  node[circle, draw, fill=black!10,inner sep=\inse, minimum width=\inci] {}
(2.73*\r,0) node[circle, draw, fill=black!10,inner sep=\inse, minimum width=\inci] {}
};
\end{tikzpicture}
\end{subfigure}
\begin{subfigure}[b]{0.15\textwidth}
\begin{tikzpicture}[scale=\scale]
 \draw[line width = \lwid, color = blue]{

(\r,\r)-- (-\r,\r)
(-\r,\r)-- (-\r,-\r)
(-\r,-\r)-- (\r,-\r)
(\r,-\r)-- (\r,\r)
};
 \draw[line width = \lwid, color = Green]{
};
\draw[line width = \cwid, color = black!50]{
(\r,\r)  node[circle, draw, fill=black!10,inner sep=\inse, minimum width=\inci] {}
(\r,-\r) node[circle, draw, fill=black!10,inner sep=\inse, minimum width=\inci] {}
(-\r,\r) node[circle, draw, fill=black!10,inner sep=\inse, minimum width=\inci] {}
(-\r,-\r)  node[circle, draw, fill=black!10,inner sep=\inse, minimum width=\inci] {}
};
\end{tikzpicture}
\end{subfigure}
\hspace*{-7.5mm}
\begin{subfigure}[b]{0.15\textwidth}
\begin{tikzpicture}[scale=\scale]
\draw[line width = \lwid, color = blue]{

(-\r,-\r)-- (\r,-\r)
(-\r,-\r)-- (0,\r)
(\r,-\r)-- (0,1*\r)
};
 \draw[line width = \lwid, color = Green]{
(2*\r,-\r)--(2*\r,1*\r)
};
\draw[line width = \cwid, color = black!50]{
(-\r,-\r)  node[circle, draw, fill=black!10,inner sep=\inse, minimum width=\inci] {}
(\r,-\r) node[circle, draw, fill=black!10,inner sep=\inse, minimum width=\inci] {}
(0,1*\r) node[circle, draw, fill=black!10,inner sep=\inse, minimum width=\inci] {}
(-\r,-\r)  node[circle, draw, fill=black!10,inner sep=\inse, minimum width=\inci] {}
(2*\r,-\r)  node[circle, draw, fill=black!10,inner sep=\inse, minimum width=\inci] {}
(2*\r,1*\r)  node[circle, draw, fill=black!10,inner sep=\inse, minimum width=\inci] {}
};
\end{tikzpicture}
\end{subfigure}
\begin{subfigure}[b]{0.15\textwidth}
\begin{tikzpicture}[scale =\scale]
 \draw[line width = \lwid, color = blue!0]{

(\r,\r)-- (-\r,\r)
(-\r,\r)-- (-\r,-\r)
(-\r,-\r)-- (\r,-\r)
(\r,-\r)-- (\r,\r)
(\r,\r)-- (-\r,-\r)
};
 \draw[line width = \lwid, color = Green!0]{
(\r,-\r)--(2.73*\r,0)
};
\end{tikzpicture}
\end{subfigure}
\caption{The 9 exceptions of \cref{c:Trace}} 
\label{f:exceptionsTrace}
\end{figure}

\section*{Acknowledgements} The author thanks David Wood for helpful discussions and valuable suggestions that improved both the content and presentation of this paper. \section*{Declarations}

\noindent \textbf{Use of artificial intelligence.} \\
\noindent Generative artificial intelligence tools were used to assist with proofreading, improving exposition, checking consistency of notation and terminology, and obtaining general feedback on the manuscript. All mathematical results, proofs, figures, classifications, and conclusions were developed and verified by the author.

\end{document}